\documentclass{article}
\usepackage{graphicx} 
\usepackage[%
  paperwidth=192mm,
  paperheight=262mm,
  vmargin={19mm,19mm},
  hmargin={13.7mm,13.7mm},
  headsep=12pt,
  footskip=12pt,
]{geometry}
\usepackage{tikz}
\usepackage[utf8]{inputenc}
\usepackage[T1]{fontenc}
\usepackage{lineno,hyperref}
\usepackage{amssymb}
\usepackage{amssymb}
\usepackage{amsthm}
\usepackage{mathtools}
\usepackage{tikz}
\usepackage{tikz-cd}
\usepackage{graphicx}
\usepackage{float}
\newtheorem{proposition}{Proposition}
\newtheorem{theorem}{Theorem}
\theoremstyle{definition} 
\newtheorem{defi}{Definition}
\newtheorem{lemma}{Lemma}
\theoremstyle{remark} 
\newtheorem{rem}{Remark}
\newtheorem*{notation}{Notation}
\newcommand{\mat}{$\textbf{Mat}_{\Bbbk}$}
\newcommand{\vect}{$\textbf{2Vec }$}
\newcommand{\id}{id}

\makeatletter
\newcommand{\xRrightarrow}[2][]{\ext@arrow 0359\Rrightarrowfill@{#1}{#2}}
\newcommand{\Rrightarrowfill@}{\arrowfill@\equiv\equiv\Rrightarrow}
\newcommand{\xLleftarrow}[2][]{\ext@arrow 3095\Lleftarrowfill@{#1}{#2}}
\newcommand{\Lleftarrowfill@}{\arrowfill@\Lleftarrow\equiv\equiv}
\newcommand{\xLleftRrightarrow}[2][]{\ext@arrow 3399\LleftRrightarrowfill@{#1}{#2}}
\newcommand{\LleftRrightarrowfill@}{\arrowfill@\Lleftarrow\equiv\Rrightarrow}
\makeatother

\usepackage[numbers]{natbib}
\usepackage{listings}
\def\tsc#1{\csdef{#1}{\textsc{\lowercase{#1}}\xspace}}
\tsc{WGM}
\tsc{QE}
\usepackage[all,2cell,cmtip]{xy}
\UseTwocells

\begin{document}
\let\WriteBookmarks\relax
\def\floatpagepagefraction{1}
\def\textpagefraction{.001}
\title{Typing Tensor Calculus in 2-Categories (I)}
\author{Fatimah Rita Ahmadi\\University of Oxford\\ f.ahmadi@imperial.ac.uk}
\date{January 2025}
\maketitle
\begin{abstract}
To formalize calculations in linear algebra for the development of efficient algorithms and a framework suitable for functional programming languages and faster parallelized computations, we adopt an approach that treats elements of linear algebra, such as matrices, as morphisms in the category of matrices, $\mathbf{Mat_{k}}$. This framework is further extended by generalizing the results to arbitrary monoidal semiadditive categories. To enrich this perspective and accommodate higher-rank matrices (tensors), we define semiadditive 2-categories, where matrices $T_{ij}$ are represented as 1-morphisms, and tensors with four indices $T_{ijkl}$ as 2-morphisms. This formalization provides an index-free, typed linear algebra framework that includes matrices and tensors with up to four indices. Furthermore, we extend the framework to monoidal semiadditive 2-categories and demonstrate detailed operations and vectorization within the 2-category of \vect introduced by Kapranov and Voevodsky.
\end{abstract}

\section{Introduction}
Algebraic programming proposes an approach to programming based on algebraic structures such as rings, fields, groups, and, more recently, categories. The aim is to incorporate these concepts at various stages of programming and in different contexts, such as concurrency or database analysis. These ideas are beautifully explained by Richard Bird and Oege de Moor in their book~\cite{bird1996algebra}. This field lies at the intersection of algebra, computer science, and logic, and it also includes formal and constructive mathematics, for this purpose check Girard, Lafont and Taylor's book~\cite{girard1989proofs}. A well-known example of such an approach is relational algebra for database queries. Another relevant example is the use of categorical constructs such as monad or functor in functional programming languages like Haskell, or proof assistants such as Isabelle~\cite{paulson2019formalising}, Coq~\cite{huet1997coq}, and Lean~\cite{moura2021lean}. With the rise of proof assistants, this approach has regained popularity, for instance see Terence Tao's paper~\cite{tao2024machine}; thus, expressing mathematical objects in terms of category theory or type theory offers a seamless transition to more formalized structures used in these proof assistants, check~\cite{moura2021lean, program2013homotopy}.

From a programming perspective, one benefit of thinking about programming in terms of category theory concepts is succinctly captured by Milewski: \textit{The essence of both is composition}, with the added benefit that category theory naturally encourages thinking about types, morphisms, and composition, aligning well with how skilled programmers should approach problem-solving~\cite{milewski2018category}. Additionally, category theory unifies all notions of functional programming into a well-defined and rigorous framework. Another area of interest focuses on formalizing mathematical entities in a more abstractly typed and point-free system for efficiency. Linear algebra is one such example. It serves as the primary toolbox in many fields, such as quantum physics, cryptography, and machine learning, to name a few. The focus of our paper is on this direction: typing elements of linear algebra within suitable categories and exploiting categorical properties and structures to perform more efficient computations.

Apart from implementing linear algebra algorithms, a recent proposal in the literature suggests a categorical framework for unifying two main approaches to deep neural networks, as discussed by Gavranovi{\'c} et.al.~\cite{gavranovicposition}. These approaches are typically described as "bottom-up" (implementation) and "top-down" (constraint satisfaction). The bottom-up approach treats neural networks at the tensorial level, using automatic differentiation packages such as TensorFlow~\cite{abadi2016tensorflow}. The paper introduces an approach based on the universal algebra of monads valued in a 2-category of parametric maps. I believe the framework presented in my paper complements their work, specifically with respect to RNNs and Appendix H, where a typing for neural network weights is introduced.

Typing matrices as morphisms in a category was noted early on as an example by Mac Lane in his book~\cite{mac2013categories}. The category of matrices, $\mathbf{Mat}$, which includes natural numbers as objects and matrices as morphisms, was a key example. This approach was further expanded by Macedo and Oliveira~\cite{macedo2013typing}, who explained the details of linear algebra algorithms, such as Gaussian elimination with a biproduct approach, within a category with finite biproducts. The cornerstone of results of my paper and Macedo and Oliveira is vectorization which is an implication of the \textbf{currying map} in a closed monoidal setting (category and 2-category). 

The vectorization is used to increase the efficiency of calculations. For example, a powerful feature of TensorFlow (a package for deep learning) and NumPy (a Python package for linear algebra) is their use of vectorization. Vectorization allows for much faster data processing compared to for-loops. These packages leverage the Single Instruction, Multiple Data (SIMD) capabilities in modern CPUs, enabling parallel computations with fewer operations. This approach also facilitates thread-level parallelism. On systems with multiple cores, libraries can distribute operations across threads, running tasks simultaneously on different units~\cite{lam1991cache}.

Another reason for faster data processing is cache friendliness. In computer systems, a cache is a small, high-speed storage area located near the CPU. Its purpose is to store frequently accessed data and instructions, enabling faster retrieval compared to accessing the main memory (RAM). Caching helps improve performance by reducing data access times. Vectorized operations typically access data in contiguous memory blocks, which is optimal for caching. In contrast, loops may access data in a less predictable pattern, leading to more cache misses~\cite{lam1991cache}.

Developing Mac Lane and Macedo and Oliveira's project to include higher-rank matrices, or tensors, motivates my paper. The term "tensor" has been widely used in the literature for various purposes. In the current context, a tensor is an umbrella term for scalars, vectors, and matrices. Scalars are tensors with no indices, vectors are tensors with one index, and matrices are tensors with two indices. For tensors with more than two indices, we collectively refer to them as tensors.

To take the first step toward a categorical typing of tensors, one must identify a suitable category. The category of matrices, \mat, which captures the structures of a semi-additive category, provides a useful intuition. By translating \mat into the language of higher categories, we observe vectors as objects, matrices as 1-morphisms, and tensors with at most four indices as 2-morphisms. This simple observation suggests that the suitable category for our purposes should be a 2-category with extra structures that allow us to recover all properties of the category \mat from its Hom-categories, i.e., a 2-category with an object. This 2-category needs to have well-defined biproducts between objects, biproducts between 1-morphisms, and additions between 2-morphisms. Moreover, these additions should be consistent and interrelated. This generalization leads us to the definition of semiadditive 2-categories.

The current literature lacks a rigorous definition of biproducts in 2-categories. In this paper, we define 2-biproducts (or biproducts in 2-categories) and describe the properties of semi-additive 2-categories. We propose both algebraic and limit-form definitions of biproducts and demonstrate their compatibility. Using this 2-category as a foundational framework, we outline the details of typing tensor calculus.

Following the same spirit as in the categorical case, where the guiding category was \mat, we need a simple example of a 2-category that possesses all the properties of our proposed definition for semiadditive 2-categories. This example will help provide intuition throughout the paper. Fortunately, such an example exists in the literature and has been studied extensively, although not specifically as a 2-category with biproducts. Kapranov and Voevodsky construct this 2-category, called $\mathbf{2Vec}$, as an example of their definition for symmetric monoidal 2-categories~\cite{kapranov19942}. The objects of $\mathbf{2Vec}$ are natural numbers, the 1-morphisms are matrices between natural numbers whose entries are finite-dimensional vector spaces, and the 2-morphisms are linear maps between vector spaces. It is widely believed that $\mathbf{2Vec}$ also serves as a prototype for semi-additive 2-categories. However, due to the lack of a rigorous definition for semi-additive 2-categories, this claim has not been closely examined. In addition, they did not spell out the details of compositions and left it to the reader. In this paper, we explain the details of compositions and tensorial operations. Similar to the categorical case, we examine the vectorization procedure and demonstrate how rich 2-categorical structures provide us with two stages of vectorization: one at the level of objects, 1-morphisms, and 2-morphisms, and the other internal to Hom-categories and at the level of 1- and 2-morphisms.

\begin{rem}
We assume a primary knowledge of category theory including definitions of category and bicategory, concepts of duality and universality. For a concise review of bicategories, consult Leinster's paper~\cite{leinster1998basic}. By a 2-category, we mean a strict 2-category: bicategories whose associators and unitors are identities.  
\end{rem}

\begin{rem}
We also assume the reader is familiar with monoidal categories introduced by Mac Lane~\cite{mac2013categories}. In the 2-categorical case, although we do not evoke the properties of monoidal 2-category explicitly, we assume such a structure while defining tensor products of higher-rank tensors. For a review on monoidal 2-categories see my paper~\cite{ahmadi2020monoidal}. 
\end{rem}

\begin{rem}
We use the pasting lemma for 2-categories introduced by Power in his work~\cite{power19902} whenever we intend to recognize equal 2-morphisms in proofs, such as in the proof of Proposition~\ref{pro-dist}. 
\end{rem}

\begin{notation}
Capital letters, $A, B, \ldots$ are reserved for objects, small letters $f, g, \ldots$ for 1-morphisms, and Greek letters $\alpha, \beta, \ldots$ for 2-morphisms. The horizontal composition is denoted by juxtaposition, and vertical composition of 2-morphisms by ``odot" $\odot$. Plus sign $+$ represents addition of morphisms in categories and addition of 2-morphisms in 2-categories. Addition (biproduct) of objects in categories and addition of 1-morphisms in 2-categories are shown by ``oplus" $\oplus$. Finally, addition of objects in 2-categories or 2-biproducts is represented by ``boxplus" $\boxplus$.

For every object $A$, the identity 1-morphism or 1-identity is shown by $id_A$, and for every 1-morphism $f$, the identity 2-morphism or 2-identity by $1_f$. Whenever it is clear from the context, in the horizontal composition of 1-identities with 2-morphisms, we leave out 1 and use the whiskering convention $1_f \circ \alpha = f \alpha$. 
\end{notation}

\begin{rem}
	If you are reading this paper in monochromatic, in Figures \ref{fig:product2} and \ref{fig:product} and Lemma \ref{lemma4}, red lines are dashed and blue lines are solid. Colors in other pictures do not play an important role.  
\end{rem}

\begin{rem} \textbf{(Important)}
Our paper should convince you that in a (semiadditive) n-category, \textbf{n-morphisms are tensors of rank $2^n$}, check Figure~\ref{fig:tensor-n-rank}.
\end{rem}
\begin{figure}[h]
\centering
\begin{tikzpicture}
    \node (0) at (0, 0) {$i$};
    \node (1) at (4, 0) {$j$}; 
    \node (2) at (2, 0.8) {};
    \node (3) at (2, -0.8) {};
    \node (4) at (1.6, 0) {};
    \node (5) at (2.4, 0) {};
    \node (6) at (8, 0) {$\aleph_{(iji'j')_\theta}^{(iji'j')_\alpha}: \theta_{ij}^{i'j'} \xRrightarrow{} \alpha_{ij}^{i'j'}$};
    \node (7) at (12, 0){$\dots$};

    \draw[->, bend right=45] (0) to node[below]{$M'_{ij}$} (1);
    \draw[->, bend left=45] (0) to node[above]{$M_{ij}$} (1);
   \draw[->, double, bend right=50] (2) to node [left] {$\theta_{ij}^{i'j'}$} (3);
    \draw[->, double, bend left=50] (2) to node [right] {$\alpha_{ij}^{i'j'}$} (3);
    \draw[-] (4) to node[above]{$\aleph$} (5);
    
\end{tikzpicture}
\caption{n-morphisms are tensors of rank $2^n$.}\label{fig:tensor-n-rank}
\end{figure}

\section{Outlines}
The paper is divided into two main sections; Section~\ref{sec:1-dimension} is mainly a review of typing matrices in categories and Section~\ref{sec:2-dim} is our work for expanding the setup to 2-categories and typing tensors. The two sections are designed parallel; meaning, in Section~\ref{sec:2-dim}, we will define the 2-categorical notions of all concepts explained in Section~\ref{sec:1-dimension}.

In the categorical case, we will discuss the definitions of biproducts in categories, establish a definition for semiadditive monoidal categories, discuss typing matrices in the category of matrices, \mat, and conclude with vectorization using the currying map. We will briefly discuss a few linear algebra algorithms in this framework summarized from~\cite{macedo2013typing}. 

In the 2-categorical case, we will define algebraic and limit-form definitions for biproducts in 2-categories and show they are equivalent definitions. We will define semiadditive 2-categories and discuss typing matrices and tensors up to four indices in this 2-category. We will give the details of tensorial operations in a 2-category called \vect and conclude by introducing a currying map and vectorization of tensors in symmetric monoidal semiadditive 2-categories.  We finish by spelling out the details of vectorization in \vect. 

\section{Typing matrices in categories: categorical case}\label{sec:1-dimension}
This section primarily follows the results of Macedo and Oliveira~\cite{macedo2013typing}; however, they do not discuss the origin of the definition of biproducts. Since we intend to define a 2-biproduct in 2-categories and to preserve the self-sufficiency of our paper, and to establish a suitable 2-categorical framework, we follow the categorical definition of biproducts step by step. The review part of semiadditive categories is comprehensive and provides enough intuition for readers when we switch to 2-categories. 
\subsection{Semiadditive categories}\label{1-d}
Semiadditive categories are the generalization of the category of matrices, \mat. They generalize and capture important structures and properties of the category of matrices, particularly, those important for typing linear algebra. They include a well-defined addition between objects which follows an important categorical notion, i.e. universality. 

We start from the definition of constant morphisms, which are essentially morphisms that behave like an absorbent whose composition with any morphism creates a constant morphism. These are rather specific morphisms which a category might not have. We define constant morphisms to define algebraic definition of biproducts in a category, in this case, all hom-sets of a category have constant morphisms. We then continue with defining zero morphisms which are absorbent from both left and right directions. We then define biproducts and reexamine category of matrices in the light of the definition of semiadditive categories and biproducts. This section is mainly based on Mac Lane's and Borceux's books~\cite{mac2013categories, borceux1994handbook}. 
\subsubsection{Constant morphisms in categories}
 A constant morphism is a specific morphism whose left and right composites for all composable morphisms are equivalent or in other words it is an absorbent morphism.  Generally, a category can have constant morphisms for some pairs of objects; however, it could be the case that not all hom-sets have constant morphisms. 

If all hom-sets have constant morphisms, they should be compatible and composed uniquely. Meaning, if we compose a constant morphism with another morphism, the resulting constant morphism should be exactly the constant morphism of the composite set. Having the definition of constant morphisms, one can define a category whose hom-sets are monoids, where constant morphisms act as the monoid identity element. Such a category with one object, defines a monoid. In Haskell, to define monoids, we have the following code in which the constant morphism has type "nothing", i.e. \textbf{mempty}. 
\begin{lstlisting}
    class Monoid m where
    mappend :: m -> m -> m
    mempty :: m
\end{lstlisting}
\begin{defi}
A morphism $f:X\longrightarrow Y$ is \textit{left constant} if for every object $W$ and every pair of morphisms $h, g: W \longrightarrow X$, the right composition with $f$ is the same, $f g = f h$. 
\end{defi}
\begin{defi}
	A morphism $f: X\longrightarrow Y$ is \textit{right constant} if for every object $Y$ and for every pair of morphisms $h, g: Y \longrightarrow Z$ , the left composition with $f$ is the same, $g f = h f$. 
\end{defi}
\begin{defi}
	A morphism is \textit{constant} if it is both left and right constant. 
\end{defi}
If all hom-sets of a category have constant morphisms, when composing constant morphisms successively, we expect the outcome of composition to be a constant morphism. The following definition indicates the compatibility condition between constant morphisms in hom-sets. 
\begin{defi} 
A \textit{locally pointed category} is a category whose hom-sets for every pair of objects $(A, B)$ have constant morphisms $\star_{B, A}:A \longrightarrow B$ if for an arbitrary pair of morphisms $A \xrightarrow{f} B\xrightarrow{g} C$, such that,
\begin{equation}\label{eq:locally-pointed}
g\star_{B,A} = \star_{C, A} = \star_{C, B}f
\end{equation}
\end{defi}
One immediate question is whether the constant morphisms are unique. Meaning, if one composes two constant morphisms, would one obtain the constant morphism of the resultant hom-set? In other words, under which condition hom-sets only have on constant morphism. The answer lies on the condition defined for locally pointed category, for a locally pointed category, it is indeed the case and the aforementioned condition, Equation~\ref{eq:locally-pointed} guarantees the uniqueness. 
\begin{proposition}\label{pro1}
In a locally pointed category, the family of constant morphisms is unique. 
\end{proposition}
\begin{proof}
Consider two different families of constant morphisms $\{\star\}$ and $\{\star'\}$. 
If $A \xrightarrow{\star_{B, A}} B \xrightarrow[\star'_{C, B}]{\star_{C, B}} C$, then from left constantness of ${\star'}$, we have $\star'_{C, B}  \star_{B, A}=\star'_{C, A}$ and from right constantness of ${\star}$, we have $\star'_{C, B}  \star_{B, A}=\star_{C, A}$ therefore, given Equation~\ref{eq:locally-pointed} $\star_{C, A}=\star'_{C, A}$ for every $(A, C)$.
\end{proof}
As you observed the constantness is a liberal structure, in the sense that it does not follow any specific categorical property such as universality. Now it is the time to define a similar definition that needs to follow universality condition. Later, we will demonstrate how and when these definitions coincide with each other. Let us switch our focus to objects, and define objects which are absorbent. 
\begin{defi}
An \textit{initial} object is an object $I$, that for all objects $A$, there exists a unique morphism $I \xrightarrow{i_{A}}A$. 
\end{defi}
Note that the uniqueness of the morphism specifies the universality condition. Co-definition of the initial object specifies the terminal object. 
\begin{defi}
An \textit{terminal} object is an object $T$, that for all objects $A$, there exists a unique morphism $A \xrightarrow{t_{T}}T$. 
\end{defi}
You can observe the reason we call these objects absorbent; first they absorb all morphisms and second when these two objects coincide we have a zero object. The morphisms that factorize through zero objects are constant. 
\begin{defi}
A \textit{zero object} is an object which is simultaneously initial and terminal. 
\end{defi}
Because the zero object is both initial and terminal, there exists a pair of unique morphisms for each object that starts and ends at the zero object.
\[
I \xrightarrow{i_A} A \xrightarrow{t_A} T
\]
If one composes these unique morphisms $t_Ai_A$ and composes the result with any other morphism the result is always the same morphism and constant due to the uniqueness of the morphism, $i_B$.  
$$
\begin{tikzpicture}
    \node (0) at (0, 0) {$I$};
    \node (1) at (1, 0) {$A$};
    \node (2) at (2, 0) {$T$}; 
    \node (3) at (3, 0) {$B$}; 

    \draw[->] (0) to node[above]{$i_A$}(1); 
    \draw[->] (1) to node[above]{$t_A$}(2); 
    \draw[->] (2) to node[above]{$t_B$}(3); 
    \draw[->, bend right=30] (0) to node[below]{$i_B$}(3); 
\end{tikzpicture}
$$
Now if the initial and terminal objects are isomorphic, the terminal and initial morphisms coincide and define zero morphisms. 
\begin{defi}
In a category with the zero object, a \textit{zero morphism} is a morphism which factorizes through the zero object. 
\end{defi}
\begin{rem}
In a category with the zero object, the family of zero morphisms is the only family of constant morphisms thanks to Proposition \ref{pro1}. 
\end{rem}
\begin{rem}
	A category might have constant morphisms but not a zero object.
\end{rem}
We are now in the right position to introduce a notion of addition between objects. We shall do so using the properties and structures laid out in the previous section. 
\subsubsection{Biproducts in categories}\label{sect:biproduct}
As mentioned earlier, a biproduct is simply a categorical addition between objects. One definition is based on products and coproducts in categories, and it satisfies a universal condition. Another definition is the algebraic definition with five equations which capture exactly the same addition albeit simpler and easier to check. In this section, we define both definitions and show how they coincide with each other. This part is rather necessary in order to make sure we have a good understanding of primary concepts and are able to define the 2-categorical correspondence.  

The algebraic definition needs a well-defined addition between morphisms, hence, an appropriate category to start with is a category whose all hom-sets are commutative monoids. We will show in the following lemma that the zero morphisms of monoids are constant morphisms which are compatible with each others and satisfies Condition~\ref{eq:locally-pointed}. 
\begin{lemma}
A category whose hom-sets are commutative monoid is a locally pointed category. 
\end{lemma}
\begin{proof}
Consider $A \xrightarrow[e_{B, A}]{f} B \xrightarrow{g} C$, for which $e$ is the unit element. From linearity of composition $g   (f + e_{B, A}) = g   f + g   e_{B, A}$ and since $e_{B, A}$ is the unit element of monoid $hom(A, B)$, so, $g   (f + e_{B, A}) = g   f$, and $g   e_{B, A}$ is the unit element of monoid $hom(A, C)$.    
\end{proof}
We are now ready to see the algebraic definition of biproducts.  Algebraic definitions, whenever they exist, are more desirable, since for instance: checking whether an algebraic definition is preserved by a particular functor, boils down to checking functionality. In terms of programming, it is also desirable since typed equations are more straightforward to follow than cone following and universal conditions for each pair of objects. 
\begin{defi}
In a category whose hom-sets are commutative monoids, a \textit{biproduct} of a pair of objects $(A, B)$ is a tuple $(A\oplus B, p_A, p_B, i_A, i_B)$ where $0$ are constant morphisms, such that:
\begin{align*}
& p_A  i_A=id_A, \hspace{0.3cm}p_B  i_B=id_B,\hspace{0.3cm}p_A  i_B=0_{A,B}, \hspace{0.3cm} p_B  i_A=0_{B,A},  \hspace{0.3cm}i_A  p_A + i_B  p_B=\id_{A\oplus B}
\end{align*}
\end{defi}
This definition is the primary definition for our objective of typing matrices; however, since we will define a corresponding biproducts in 2-categories. We need to see the categorical properties which help us to have a better grasp of biproducts in categories. This definition is based on products and coproducts. Products is similar to pairing of objects or types and coproducts is a sort of co-pairing. In short, a pair of objects in a category can have non-isomorphic products and coproducts, but if they are isomorphic, we say this pair of objects has a biproduct. 
\begin{defi}
In a category, a \textit{product} of pair of objects $(A, B)$ is a tuple $(A \times B, p_A, p_B)$ such that for an arbitrary cone $(X, f: X \longrightarrow A, g: X \longrightarrow B)$ shown in Figure \ref{fig:product2}, there exists a unique morphism $b: X \longrightarrow A \times B$ which satisfies $p_A b = f$ and $p_B b = g$. Morphisms $p_A$ and $p_B$ are called projections. 
\begin{figure}[ht!]
	\centering
	\begin{tikzpicture}
		\node [] (0) at (-2, -1) {$A$};
		\node [] (1) at (2, -1) {$B$};
		\node [] (2) at (0, 1) {$A \times B$};
		\node [] (3) at (0, 3) {$X$};
		
		\draw [->, red, dashed]  (3)  to [ bend right=40] node[left]{~$f$} (0);
		\draw [->] (2) to [bend right=20]  node[below right] {$p_A$} (0);
		\draw [->, red, dashed]  (3)  to [bend left=40] node[right]{~$g$}(1);
		\draw [->] (2) to [bend left=20] node[below left] {$p_B$}(1);
		\draw[red, ->, dashed] (3) to node[left] {$ b$}(2); 
	\end{tikzpicture}
	\caption{Product of $(A, B)$.}\label{fig:product2}
\end{figure}
\end{defi} 
\begin{defi}
In a category, a \textit{coproduct} of pair of objects $(A, B)$ is a tuple $(A \sqcup B, i_A, i_B)$ such that for an arbitrary cone $(X, f: A \longrightarrow X, g: B \longrightarrow X)$ there exists a unique morphism $b: A \sqcup B \longrightarrow X$ which satisfies $bi_A = f$ and $bi_B= g$. Morphisms $i_A$ and $i_B$ are called injections. (A cone similar to Figure~\ref{fig:product2} can be drawn although with reversed arrows.)
\end{defi} 
\begin{defi}\label{defi2.1}
In a locally pointed category, the \textit{canonical morphism} between a coproduct of a pair of objects $A_1, A_2$, i.e. $A_1 \sqcup A_2$ and a product $A_1 \times A_2$ is morphim $r$ such that it satisfies $p_k   r   i_j=\delta_{k, j}$.
\begin{equation}\label{equation:biproduct}
A_j \xrightarrow{i_j} A_1 \sqcup A_2 \xrightarrow{r} A_1 \times A_2 \xrightarrow{p_k} A_k, \hspace{1cm}\text{if } \hspace{0.2cm} i, j \in \{1, 2\}
\end{equation}
\end{defi}
\begin{defi}\label{defi1}
In a locally pointed category, a pair of objects $A$ and $B$ has a \textit{biproduct} if the canonical morphism $r$ is an isomorphism. 
\end{defi}
Now we can clearly demonstrate how these two definitions are equivalent, or how one can recover the algebraic definition from Equation \ref{equation:biproduct} assuming $r$ is an isomorphim. 
\begin{defi}
A \textit{semiadditive category} is a category with finite biproducts.
\end{defi}
It is worth noting that if a category has biproducts of the second form, i.e. equivalent products and coproducts; it imposes a monoid structure on hom-sets. That, in turn, will result in the algebraic definition of biproducts. 
\begin{proposition}\label{pro-26}
A semiadditive category, with the definition based on (co-)product is a category whose hom-sets are commutative monoids.  
\end{proposition}
\begin{proof}
Define a monoid addition in $\hom(A, B)$ as below:
$$
	f+g: A \xrightarrow{\bigtriangleup_A} A \oplus A \xrightarrow{f\oplus g} B\oplus B \xrightarrow{\bigtriangledown_B} B
$$
Such that $\bigtriangleup_A=\begin{bmatrix}
	id_A \\ id_A
\end{bmatrix}$  and $\bigtriangledown_B = \begin{bmatrix}
id_B & id_B
\end{bmatrix}$. 
Commutativity and associativity of $+$ follow from associativity and commutativity of biproduct $\oplus$:
\begin{align*}
& f+g = \bigtriangledown_B (f \oplus g) \bigtriangleup_A = \bigtriangledown_B (g \oplus f) \bigtriangleup_A =  g+f \\
& (f+g)+h = \bigtriangledown_B p_{12}((f \oplus g) \oplus h)i_{12} \bigtriangleup_A =   \bigtriangledown_B p_{12}(f \oplus (g \oplus h))i_{12} \bigtriangleup_A =  f+(g+h)
\end{align*}
\end{proof}
As expressed in~\cite{macedo2013typing, bird1996algebra}, one can interpret biproduct expressions as usual programming elements. That is $[f | g]$ is "f junc g" and $\begin{bmatrix}
k \\
-\\
l
\end{bmatrix}$ is "k split l". In the following, we will see a more explicit treatment of biproducts with matrices. 

Before describing the main example of this category, we need to define monoidal categories. Although in terms of category of matrices, this is simply multiplication of natural numbers, it will be necessary in the context of \textbf{currying} and vectorization. We will give the definition and a few examples; one can further check the structures and properties of this category in \cite{mac2013categories}. 
\begin{defi}[Monoidal Category]
A \textit{monoidal category} $\mathcal{C}$ is a category with a bifunctor, $\otimes: \mathcal{C} \times \mathcal{C} \longrightarrow \mathcal{C}$ such that for every three objects, there exists a (natural) isomorphism $a_{A, B, C}$ such that it satisfies the pentagonal and tiangle equations given in Page 23 of Joyal and Street's paper~\cite{joyal1993braided}. 
$$
a_{A, B, C}: (A \otimes B) \otimes C \longrightarrow A \otimes (B \otimes C)
$$ 
\end{defi}
In category of sets, \textbf{Set}, a straightforward choice of such a bifunctor is Cartesian product. In category of vector spaces, monoidal product is the tensor product of vector spaces. As we will see in the next section, in category of matrices, \mat, it is the multiplication of numbers. 
\begin{rem}
In a monoidal and semiadditive category, biproduct and monoidal products should behave coherently with respect to each other, namely, the monoidal product distributes over biproducts. 
$$
A \otimes (B \oplus C) \cong (A \otimes B) \oplus (A \otimes C)
$$
\end{rem}

\subsubsection{Matrices and biproducts}
Consider the category of matrices over a field, \mat. Objects of this category are natural numbers, and morphisms are matrices whose entries belong to the field. For example, \textbf{hom}(2, 3) consists of all 3 by 2 matrices over Field $\Bbbk$. The composition of morphisms is matrix multiplication. 
\begin{align*}
    & 2 \xrightarrow{M_{3\times 2} \in hom(2, 3)} 3, & 2 \xrightarrow{M_{3\times 2} \in hom(2, 3)} 3 \xrightarrow{N_{3\times 5} \in hom(3, 5)} 5
\end{align*}
One can immediately observe the advantage of using categorical composition instead of usual matrix multiplication; point-free typed matrices allow us to avoid checking the dimensions of a matrix for matrix multiplication. In most  software such as MATLAB or programming languages such as python, the programmer needs to take care of shape and dimensions of matrices while defining them; otherwise, it raises an error and does not calculate the multiplication. This method of defining the matrices, automatically will consider the types of matrices. 

Matrix transposition is another linear algebra operation which is cumbersome. As one needs to swap all rows and columns, but in the categorical form, we only reverse the arrow; $M^T: 3 \longrightarrow 2$. Later we also see how we can write it down with only using \textbf{vec} and \textbf{unvec} maps. 

Since hom-sets with matrix addition form abelian groups, this category is furthermore abelian, which incidentally implies semiadditivity. To see the role of biproducts in matrix calculus, take the example below \ref{fig:mat-biproduct} in which $b = \begin{bmatrix}
	f \\
	\hline 
	g
\end{bmatrix}$ and $ b'= \begin{bmatrix}
	h &
	\vline 
	& k
\end{bmatrix}$.
\begin{figure}\label{fig:mat-biproduct}
\centering
	\begin{tikzpicture}
		\node (0) at (-2.5, 0){3};
		\node (1) at (0, 0){3+2};
		\node (2) at (2.5, 0){2};
		\node (3) at (0, 2.5){1};
		
		\draw[->] (1) to node[above]{$p_1$} (0);
		\draw[->] (1) to node[above]{$p_2$} (2);
		\draw[->] (3) to node[right]{$b$} (1);
		\draw[->, bend left=40, blue] (1) to node[left]{$b'$} (3);
		\draw[->] (3) to node[above left]{$f$}  (0);
		\draw[->] (3) to node[above right]{$g$} (2);
		\draw[->, bend left=40, blue] (0) to node[left]{$h$}  (3);
		\draw[->, bend right=40, blue] (2) to node[right]{$k$} (3);
		\draw[->, bend left = 30, blue] (2) to node[below]{$i_2$}(1);
		\draw[->, bend right = 30, blue] (0) to node[below]{$i_1$}(1);
	\end{tikzpicture}
    \caption{Matrices and biproducts.}
\end{figure}
Projections are $p_1 = \begin{bmatrix}
	1 & 0
\end{bmatrix}$ and $p_2 = \begin{bmatrix}
	0 & 1
\end{bmatrix}$ and injections $i_1 = \begin{bmatrix}
	1 \\ 0
\end{bmatrix}$ and $i_2 = \begin{bmatrix}
	0 \\ 1
\end{bmatrix}$. We have:
\begin{align*}
	& p_1 \begin{bmatrix}
		f \\
		\hline 
		g
	\end{bmatrix} = \begin{bmatrix}
		1 & 0
	\end{bmatrix} \begin{bmatrix}
		f \\
		\hline 
		g
	\end{bmatrix} = f, && \begin{bmatrix}
		h &
		\vline 
		& k
	\end{bmatrix} i_1  = \begin{bmatrix}
		h &
		\vline 
		& k
	\end{bmatrix} \begin{bmatrix}
		1 \\ 0 
	\end{bmatrix} = h  \\
	& p_2 \begin{bmatrix}
		f \\
		\hline 
		g
	\end{bmatrix} = \begin{bmatrix}
		0 & 1
	\end{bmatrix} \begin{bmatrix}
		f \\
		\hline 
		g
	\end{bmatrix} = g &&\begin{bmatrix}
		h &
		\vline 
		& k
	\end{bmatrix} i_2  = \begin{bmatrix}
		h &
		\vline 
		& k
	\end{bmatrix} \begin{bmatrix}
		0 \\ 1 
	\end{bmatrix} = k
\end{align*}
We can verify the conditions of the biproduct with the explicit matrix representation of projections and injections. For example,  
\begin{align*}
& i_1 p_2 = 0, p_1i_1 + p_2i_2 = \begin{bmatrix} 1 & 0 \\ 0 & 1 \end{bmatrix}, p_1i_1 =\begin{bmatrix} 1 & 0 \\ 0 & 0 \end{bmatrix} 
\end{align*}
This partitioning of matrices via biproducts is utilized by \cite{macedo2013typing} to implement divide-and-conquer and Gaussian elimination algorithms. We don't repeat their contribution but emphasis, this will be carried out in categorical and 2-categorical levels. When we further define biproducts in 2-categories, all categorical constructions will be automatically carried out in Hom-categories. In terms of 2-categorical construct, one can implement these results locally in each Hom-category. 

As it should be clear from the explicit matrix form of projections and injections above, projections and injections play the role of basis elements which project/inject the entries of morphisms to/from biproducts. For instance, assume the following list of morphisms $$f: A_1 \longrightarrow A'_1,  g: A_1 \longrightarrow A'_2,  h: A_1 \longrightarrow A'_3$$ $$f': A_2 \longrightarrow A'_1, g': A_2 \longrightarrow A'_2, h': A_2 \longrightarrow A'_3 $$ 
It is not difficult to show that $A_1 \oplus A_2 \xrightarrow{t} A'_1 \oplus A'_2 \oplus A'_3$ is $t =\begin{bmatrix}
f & f'\\
g & g' \\
h & h'
\end{bmatrix}$, when projections are $p_1 = \begin{bmatrix}
1 & 0 & 0
\end{bmatrix} $, $p_2 = \begin{bmatrix}
0 & 1 & 0
\end{bmatrix} $, $p_3 = \begin{bmatrix}
0 & 0 & 1
\end{bmatrix} $ and injections are $i_1 =  \begin{bmatrix}
	1 \\
	0
\end{bmatrix} $, $i_1 =  \begin{bmatrix}
	0 \\
	1
\end{bmatrix}$. 

Category of matrices, \mat, provide an intuition for us to guess the correct order/rank of matrices when working with general semiadditive categories. Henceforth, we represent morphisms in semiadditive categories in matrix notation when it eases our calculations.  
\subsection{Typed algorithms and vectorization}\label{section:vectorization1}
The authors of \cite{macedo2013typing} summarize the steps of matrix calculus algorithms such as divide-and-conquer and Gaussian elimination using semiadditive category properties. This, however, follows from three basic procedures; 1.properties of semiadditive categories, 2.different variants of biproducts, and 3.vectorization. One variant of biproduct, for instance, makes the elimination of rows possible and adds a minus sign for row and column deduction. Without repeating the results of the paper, we describe divide-and-conquer algorithm as an instance. For other algorithms, check their paper~\cite{macedo2013typing}.

The goal is to multiply these two block matrices as follows:
\begin{align*}
    & \begin{bmatrix}
        f | g
    \end{bmatrix}. \begin{bmatrix}
        h \\ - \\ k
    \end{bmatrix} = f.h \oplus g.k
\end{align*} 
We first use the definition of biproduct to write the second blocked column matrix:
\begin{align*}
    & \begin{bmatrix}
        f | g
    \end{bmatrix} . \begin{bmatrix}
        h \\ - \\ k
    \end{bmatrix} = 
    \begin{bmatrix}
        f | g
    \end{bmatrix} . (i_1.h \oplus i_2.k)
\end{align*}
We then use bilinearity of semiadditive categories and injection of the first and second block :
\begin{align*}
    & \begin{bmatrix}
        f | g        
        \end{bmatrix}.i_1.h \oplus \begin{bmatrix}
        f | g        
        \end{bmatrix}.i_2.k =f.h \oplus g.k
\end{align*}
Note that we did not use matrix entries or the usual product (rows times columns) here. We used the fact that these are typed morphisms plus properties of injections and projections. 

Another subroutine is vectorization which converts a matrix to a column vector. This step is a familiar step of some deep learning algorithms; for instance, in convolutional neural networks, vectorization refers to the process that transforms the original data structure into a vector representation, See~\cite{ren2015vectorization}. Although having a considerable data-size and powerful GPU are primary reason behind the success of these networks, vectorization deserves some credit as it makes the parallel computation possible. 

If one uses a functional programming language combined with a CNN algorithm, one can vectorize matrices in typed level. Apart from this familiar example, various reasons have been suggested in the literature in different contexts for vectorization; from hardware structure and process point of view, data structure and parsing, efficient and speedy programming without for-loops, machine learning, text processing and large language models. We summarized a few reasons such as using the architecture of modern CPU's and cash-friendliness in the introduction, and for further read, See~\cite{lam1991cache}. Our concern here is not about applications. However, given that the backbone of most of these areas is linear algebra the introduced categorical notations are straightforwardly extendable to all contexts.

In categorical language, there is a specific categorical procedure for vectorization; namely \textbf{currying}.  Here, we only need to understand the operational meaning of currying for vectorization. To review the details of currying, and its implications for formalized mathematics, for an approach suitable for programmers, consult \cite{milewski2018category} and a more abstract view in terms of formalized mathematics and Homotopy type theory check \cite{program2013homotopy}.

Let us consider an example, given a $2\times 2$ matrix $A$, with vectorization we can write it as a $4 \times 1$ column vector, $\textbf{vec}f$. 
\begin{align*}
    & f = \begin{bmatrix}
        f_{11} & f_{12} \\ f_{21} & f_{22}
    \end{bmatrix}, & \textbf{vec} f = \begin{bmatrix}
        f_{11} \\
        f_{21}\\
        f_{12}\\
        f_{22}
    \end{bmatrix}
\end{align*}
It is easy to define \textbf{vec} here as:
$$
\textbf{vec}:: (2 \leftarrow{f} 2 \times 1) \leftarrow{} (2 \times 2 \xleftarrow{\textbf{vec}f} 1) 
$$
The general form of the operation, if considered in the category of matrices, \mat is as follows:
\begin{equation}\label{eq:vect-unvect}
\textbf{vec}_K::(N \xleftarrow{f} K \times M) \rightarrow{} (K \times N \xleftarrow{\textbf{vec}f} M)  
\end{equation}
$k$ is called the ``thinning factor'' as it determines the size of vectorization. For example, in a CNN, you compress a block matrix to a $1 \times 1$ matrix, then you vectorize and you can even further compress that vector. 
\begin{figure}[ht!]
    \centering
    \includegraphics[scale=0.5]{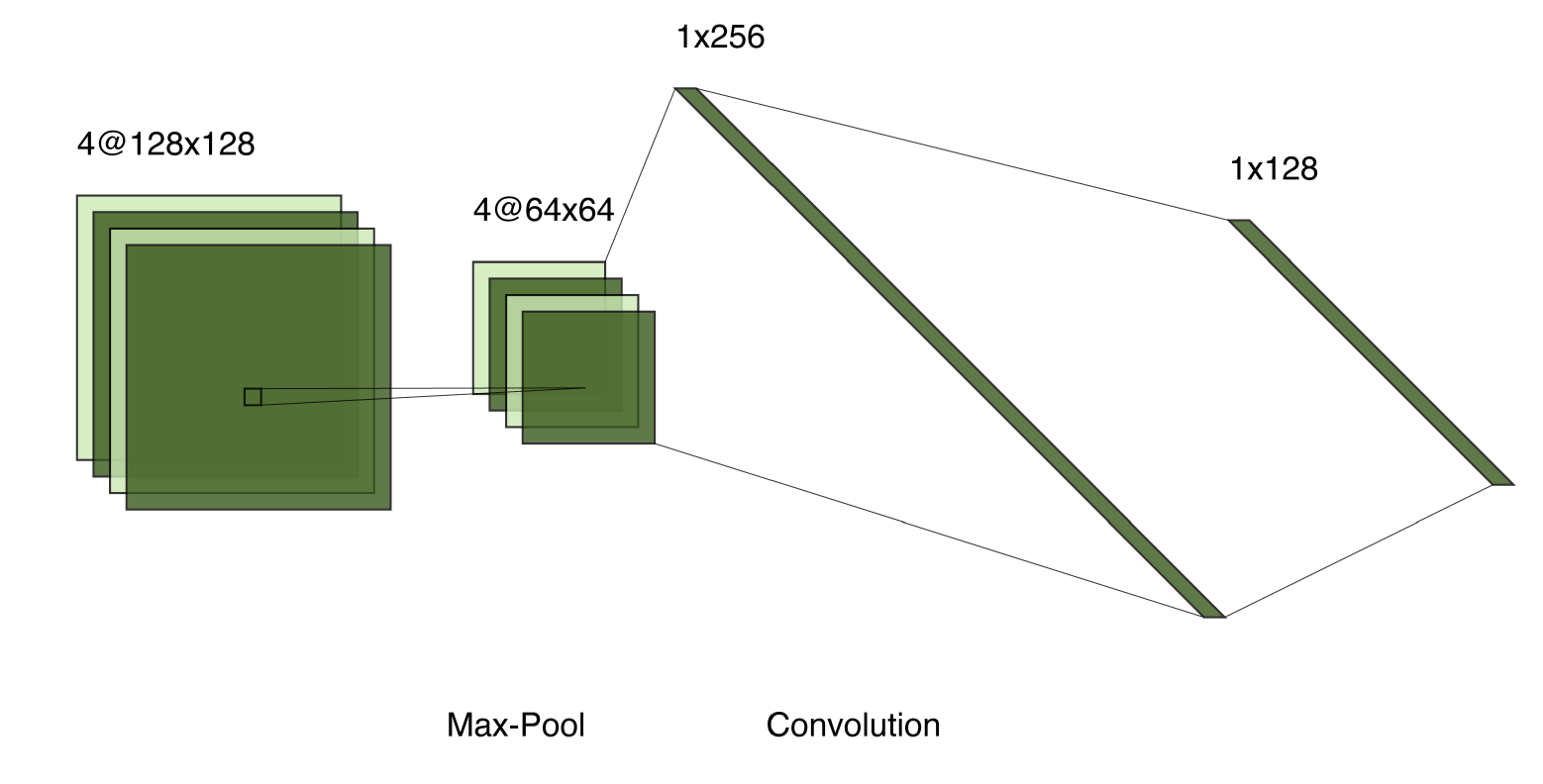}
    \caption{Vectorization with different thinning factors in layers of a convolutional neural network. }
\end{figure}

As it should be apparent from the above equation~\ref{eq:vect-unvect}, one can equally define an inverse of the transformation, \textbf{vect}, called, \textbf{unvec}.  But it is more than reversing an arrow, seeing this operation as \textbf{currying} it satisfies the universal property coming for free with categorical definition. In categorical setting, it is an isomorphisms between hom-set which assigns to a morphism $f: A \otimes B \longrightarrow C$ a morphism $g: A \longrightarrow C^B$. $C^B$ is called an exponential object, it is basically an object plus a morphism attached to the object such that it satisfies a universal property. We do not delve into the general definition, however, let us see what it means in terms of category of matrices, \mat, and transformations \textbf{vec} and \textbf{unvec}. If you check the Diagram~\ref{diagram:vect-unvect}, you can see that de-vectorization is carried out by morphism $e_k$, which assign to Morphism, $v$, Morphism, $f$. Note that $v$ is already a vector of $f$. The diagram, furthermore, shows the universal property of \textbf{currying}.
\begin{equation}\label{eq:curry}
    v = \textbf{vec}_K f \Leftrightarrow  f = e_K (id_K \otimes f) 
\end{equation}
\begin{equation}\label{diagram:vect-unvect}
 \begin{tikzcd} 
K \times N &&& K \times (K \times N) \arrow[r, "e_K"] &N\\
M \arrow[u, "v"] &&& K \times M \arrow[u, "id_K \otimes v"] \arrow[ru, "f"] &
\end{tikzcd}    
\end{equation}
Having the vectorization procedure, one can define matrix operations and algorithms completely in a vectorized fashion. Let us review transpose operation given by \cite{macedo2013typing}. Having a matrix $N \xleftarrow{f} M$, we start by first de-vectorization, $1 \xleftarrow{g=\textbf{vec}f} N \times M$, then vectorization, $N \times M \xleftarrow{h=\textbf{vec} g} 1$ and finally, $M \xleftarrow{\textbf{unvec} h} N$. Hence, 
$$
f^T = \textbf{vec}(\textbf{vec} (\textbf{unvec} f)))
$$
We close the review part and the categorical level of typing linear algebra here, and in the next sections, we look at the similar 2-categorical properties and structures.   
\section{Typing tensors in 2-categories: 2-categorical case}\label{sec:2-dim}
In this section, we extend the constructions from the categorical case to define 2-categorical analogues, aiming to establish 2-biproducts for 2-categories. To capture the tensor product between tensors, a 2-category needs to have a monoidal structure, and to define vectorization and a version of currying map, the 2-category furthermore has to be symmetric monoidal. Hence, we will establish a suitable framework is a symmetric monoidal semiadditive 2-category whose Hom-categories are symmetric monoidal and semiadditive.   We will proceed with an example that satisfies the conditions of our definition, namely $\mathbf{2Vec}$. Finally, we will conclude with a discussion on vectorization for 1- and 2-morphisms. As far as we are aware, all notions (from categorical perspective and applications) presented in this section are novel and have not been explored in the literature.
\subsection{Semiadditive 2-categories}
Intuitively, a semiadditive 2-category should possess additions for objects, 1-morphisms, and 2-morphisms. Analogous to semiadditive categories, we begin with a 2-category that includes an addition for 1-morphisms. However, in the 2-categorical setting, defining such an addition introduces intricacies, as multiple options arise for combining 1-morphisms. Our goal is to recover semiadditive categories, the most natural choice is the biproduct structure for 1-morphisms, as defined in the previous section~\ref{sect:biproduct}.

Therefore, we define a semiadditive 2-category as a 2-category whose Hom-categories have biproducts for 1-morphisms, in accordance with the definition provided in Section~\ref{sect:biproduct}. This section focuses on developing these structures. Following Section~\ref{1-d}, we begin by introducing constant morphisms. The first difference with categorical setting emerges here. Due to the presence of 2-morphisms, we also establish a notion of constantness for both 1- and 2-morphisms.

\subsubsection{Constant morphisms in 2-categories}
In principle, a 2-category can have constant 1- or 2-morphisms; we start from the top level, i.e. 2-morphisms. We define constant 2-morphisms independently for horizontal and vertical compositions. Because it is possible for Hom-categories to be locally pointed with respect to the vertical composition of 2-morphisms. It is also conceivable to have a 2-category with horizontally constant 2-morphisms with respect to horizontal composition. These two sets of constant 2-morphisms are generally independent from each other. However, we demonsterate if a 2-category happens to have two different families of horizontally and vertically constant 2-morphisms, then they collapse, due to the interchange law. Meaning, in this case, we only have one family of constant 2-morphisms. Note that the interchange law exists in 2-category theory and enforces a consistency condition between horizontal and vertical compositions, check Adámek and Rosický's review~\cite{adamek1999interchange}. 
\begin{defi}
	A 2-morphism $*_{h, k}: k\Rightarrow h$ is \textit{left constant} if for every object $A$, every pair of 1-morphisms $f, g$, and every pair of 2-morphisms, $\xi, \gamma: f \Rightarrow  g$, horizontal composition on the right is the same, $*_{h, k} \xi = *_{h, k}  \gamma$. 
	\begin{center}
		\begin{tikzpicture}
			\node (1) at (0, 0) []{$A$};
			\node (2) at (2, 0) []{$B$};
			\node (3) at (4, 0) []{$C$};
			\node (4) at (1, 0.5) []{};
			\node (5) at (1, -0.5) []{};
			\node (6) at (3, 0.5) []{};
			\node (7) at (3, -0.5) []{};
			
			\draw[->, bend left=45] (1) to node[above]{$f$}(2);
			\draw[->, bend right=45] (1) to node[below]{$g$}(2);
			\draw[->, bend left=45] (2) to node[above]{$k$}(3);
			\draw[->, bend right=45] (2) to node[below]{$h$}(3);
			\draw[->, double, thin, bend left=30] (4) to node[right]{$\xi$}(5);
			\draw[->, double, thin, bend right=30] (4) to node[left]{$\gamma$}(5);
			\draw[->, double] (6) to node[right]{$*_{h, k}$}(7);
		\end{tikzpicture}
		\label{hor-1}
	\end{center}  
\end{defi}
\begin{defi}
	A 2-morphism $*_{h, k}: k\Rightarrow h$ is \textit{right constant} if for every object $D$, every pair of 1-morphisms $l, m$, and every pair of 2-morphisms, $\alpha, \beta: l \Rightarrow  m$, horizontal composition on the left is the same, $\beta  *_{h, k} = \alpha  *_{h, k}$.
	\begin{center}
		\begin{tikzpicture}
			\node (B) at (2, 0) []{$B$};
			\node (C) at (4, 0) []{$C$};
			\node (D) at (6, 0) []{$D$};
			\node (1) at (1, -0.5) []{};
			\node (2) at (1, 0.5) []{};
			\node (3) at (3, -0.5) []{};
			\node (4) at (3, 0.5) []{};
			\node (5) at (5, -0.5) []{};
			\node (6) at (5, 0.5) []{};

			\draw[->, bend left=45] (B) to node [above]{$k$} (C);
			\draw[->, bend right=45] (B) to node [below]{$h$} (C);
			\draw[->, bend left=45] (C) to node [above]{$l$} (D);
			\draw[->, bend right=45] (C) to node [below]{$m$} (D);
			\draw[->, double] (4) to node [right]{$*_{h, k}$} (3);
			\draw[double,thin,->, bend right=45] (6) to node [left]{$\alpha$} (5);
			\draw[double, ->, bend left=45] (6) to node [right]{$\beta$} (5);
		\end{tikzpicture}
		\label{dis-2}
		
	\end{center}
\end{defi}
\begin{defi}
In a 2-category, a \textit{horizontally constant 2-morphism} is a 2-morphism which is horizontally left and right constant.
\end{defi}
We will define similar notions for vertical compositions. As you might expect, the constantness should be with respect to the vertical compositions of 2-morphisms, i.e. up and down.
\begin{defi}
	A 2-morphism $*_{k, g}: k\Rightarrow g$ in a 2-category is \textit{down constant} if for every 1-morphism $f$, and every pair of 2-morphisms $\beta, \alpha: f \Rightarrow g$, vertical composition from above is the same, $*_{k, g} \odot ~\beta = *_{k, g} \odot ~\alpha $.
\end{defi}

\begin{center}
	\begin{tikzpicture}
		\node (0) at (0, 0) []{$A$} ;
		\node (1) at (3, 0) []{$B$};
		\node (2) at (1.5, 1) []{};
		\node (3) at (1.5, 0.1) []{};
		\node (4) at (1.5, -0.1) []{};
		\node (5) at (1.5, -1.1) []{};
		\node (6) at (2.6, -0.3)[]{$g$};

		\draw[->, bend left=70] (0) to node[above]{$f$}(1);
		\draw[->] (0) to node[below]{}(1);
		\draw[->, bend right=70] (0) to node[below]{$k$} (1);
		\draw[->, double, thin, bend right=40] (2) to node[left]{$\beta$}(3);
		\draw[->, double, thin, bend left=40] (2) to node[right]{$\alpha$}(3);
		\draw[->, double, thin] (4) to node[right]{$*_{k, g}$}(5);
		
	\end{tikzpicture}
	\label{ver1}
\end{center}
\begin{defi}
A 2-morphism $*_{k, g}: k\Rightarrow g$ in a 2-category is \textit{up constant} if for every 1-morphism $f$, and every pair of 2-morphisms, $\xi, \gamma: k \Rightarrow h$, vertical composition from below is the same, $\xi \odot *_{k, g}  = \gamma \odot *_{k, g} $.
\end{defi}
\begin{center}
	
	\begin{tikzpicture}
		\node (0) at (0, 0) []{$A$} ;
		\node (1) at (3, 0) []{$B$};
		\node (2) at (1.5, 1) []{};
		\node (3) at (1.5, 0.1) []{};
		\node (4) at (1.5, -0.1) []{};
		\node (5) at (1.5, -1) []{};
		\node (6) at (2.4, -0.2)[]{$k$};

		\draw[->, bend left=70] (0) to node[above]{$g$}(1);
		\draw[->] (0) to node[below]{}(1);
		\draw[->, bend right=70] (0) to node[below]{$h$} (1);
		\draw[->, double, thin, bend right=40] (4) to node[left]{$\xi$}(5);
		\draw[->, double, thin, bend left=40] (4) to node[right]{$\gamma$}(5);
		\draw[->, double, thin] (2) to node[right]{$*_{k, g}$}(3);
		
	\end{tikzpicture}
	\label{ver-2}
\end{center}
\begin{defi}
	A \textit{vertically constant 2-morphism} is a 2-morphism which is up and down constant.
\end{defi}
As explained for the case of categories, once dealing with constant morphisms in a category with the zero object, one can consider zero morphisms (or morphisms factorizing through the zero object) as constant morphisms.  However in 2-categories, we have two sets of such constant morphisms. It is also not  trivial if zero morphisms are constant morphisms. We now present the proposition that first demonstrates the consistency of vertical and horizontal compositions and secondly guarantees zero 2-morphisms in Hom-categories are not only vertically but also horizontally constant. 
\begin{proposition}\label{pro3}
	If a 2-category has a family of horizontally constant 2-morphisms $\{*\}$ and a family of vertically constant 2-morphisms $\{*'\}$, then they are equal. 
\end{proposition}
\begin{proof} Given 1- and 2-morphisms shown in the following figure,
	\begin{center}
		\begin{tikzpicture}[scale=0.9]
			\node (1) at (0, 0) []{$A$};
			\node (2) at (3, 0) []{$B$};
			\node (3) at (6, 0) []{$C$};
			\node (4) at (0.4, 0.2)[]{$g$};
			\node (5) at (1.3, 0.9) []{};
			\node (6) at (1.3, 0.1)[]{};
			\node (7) at (1.3, -0.1)[]{};
			\node (8) at (1.3, -0.9)[]{};
			\node (9) at (4.3, 0.9) []{};
			\node (10) at (4.3, 0.1)[]{};
			\node (11) at (4.3, -0.1)[]{};
			\node (12) at (4.3, -0.9)[]{};
			\node (13) at (3.6, 0.25)[]{$j$};
			
			\draw[->, bend left=70] (1) to node[above]{$f$} (2);
			\draw[->] (1) to node[]{}(2);
			\draw[->, bend right=70] (1) to node[below]{$h$}(2);
			\draw[->, bend left=70] (2) to node[above]{$k$}(3);
			\draw[->] (2) to node[]{}(3);
			\draw[->, bend right=70] (2) to node[below]{$l$}(3);
			\draw[->, double, thin] (5) to node[right]{$*'_{g, f}$}(6);
			\draw[->, double, thin] (7) to node[right]{$\alpha$}(8);
			\draw[->, double, thin] (9) to node[right]{$\beta$}(10);
			\draw[->, double, thin] (11) to node[right]{$*_{l, j}$}(12);
			
		\end{tikzpicture}
		\label{proof1}
	\end{center}
	Due to the interchange law in 2-categories, we expect to have:
	\begin{equation} \label{eq:eq1}
		(*_{l, j} \odot \beta ) ~ (\alpha \odot {*'}_{g, f})=({*}_{l, j}  ~\alpha) \odot (\beta ~ {*'}_{g, f})
	\end{equation} 
Because $*^\prime$ is vertically constant. LHS:
	$(*_{l, j} \odot \beta)  *^{\prime}_{h, f}$. 
    
Because $*$ is horizontally constant. RHS: $*_{lh, jg} \odot (\beta  *^{\prime}_{g, f})$. 
	
Now in LHS and RHS, let $B=C$ and $k=j=l=id_B$. Therefore, we have:
	\begin{align*}
		LHS: & (*_{id_B, id_B} \odot 1_{id_B})  *^{\prime}_{h, f}=*_{h, f} & & 
		RHS:  *_{h, g} \odot (1_{id_B}  *^{\prime}_{g, f})=*^{\prime}_{h, f}
	\end{align*}
	Thus, for every pair of 1-morphisms $(h, f)$, horizontally and vertically constant 2-morphisms are equal. 
\end{proof}
\begin{rem} \textbf{(Important)}
Generally, the 2-identity $1_f$ for vertical composition is not also a horizontal 2-identity, unless $f$ is a 1-identity. That is, $1^{horizontal}_{id_B} = 1^{vertical}_{id_B}$. Because consider composable $\alpha: g \Rightarrow h$ and $1_f$, then $\alpha$ and $\alpha f: gf \Rightarrow hf$ do not have even the same source and target, Check Section B.1.1. Definition (2-category), Page 421 of Riel and Verity's book~\cite{riehl2022elements}. We use this fact throughout our paper. 
\end{rem}  
Due to the presence of 2-morphisms in 2-categories, constant 1-morphisms have a richer structure as they can be weakened to be constant up to a 2-morphism. We define the weak version here, the strict case is obtained when weakening 2-morphisms are identities. 
\begin{defi}
A 1-morphism $f: X\longrightarrow Y$ is \textit{weakly left constant} if for every object $W$ and for every pair of 1-morphisms, $h, g: W \longrightarrow X$, there exists a 2-isomorphism such that $\eta: f  g \Rightarrow f  h$. 
\end{defi}
\begin{defi}
A 1-morphism $f: X\longrightarrow Y$ is \textit{weakly right constant} if for every object $Y$ and for every pair of 1-morphisms, $h, g: Y \longrightarrow Z$, there exists a 2-isomorphism such that $\eta': g f \Rightarrow h f$. 
\end{defi}
\begin{defi}
A 1-morphism is \textit{weakly constant} if it is both weakly left and right constant. 
\end{defi}
\begin{defi}
A \textit{2-category with constant 1-morphisms} is a 2-category for which Diagram~\ref{fig1.3} commutes up to a 2-isomorphism. 
	\begin{figure}[ht!]
    \centering
		\begin{tikzpicture}
			\node (0) at (0, 0) []{$A$};
			\node (1) at (2, 0) []{$B$};
			\node (2) at (0, -2) []{$B$};
			\node (3) at (2, -2) []{$C$};
			\node[rotate = 45] (4) at (0.7, -1.5){$\Rightarrow$};
			\node[rotate=-135] (6) at (1.5, -0.7)[]{$\Rightarrow$};
			\node (8) at (0.8, -0.3)[]{$\star_{C, A}$}; 
			
			\draw[->] (0) to node[above]{$\star_{B, A}$}(1);
			\draw[->] (1) to node[right]{$g$}(3);
			\draw[->] (0) to node[left]{$f$}(2);
			\draw[->] (2) to node[below]{$\star_{C, B}$}(3);
			\draw[->] (0) to node[right]{}(3);
		\end{tikzpicture} 
        \caption{Condition of weakly constant 1-morphisms in 2-categories. }\label{fig1.3}
	\end{figure}
\end{defi}
A proposition similar to Proposition~\ref{pro1} for categories considers the uniqueness of constant 1-morphisms up to 2-isomorphisms. Because the condition above Figure~\ref{fig1.3} enforces such.  
\begin{proposition}\label{pro2}
In a 2-category with constant 1-morphisms, the family of constant 1-morphisms is unique up to a 2-isomorphism. 
\end{proposition}
\begin{proof}:
    Consider two different families of constant 1-morphisms with their characteristic weakening 2-isomorphisms, $\{\star, \eta\}$ and $\{\star', \beta\}$. Given $\eta^{-1}: \star \Rightarrow \star  \star'$ and $\beta: \star  \star' \Rightarrow \star'$, we have $\beta \odot \eta^{-1}: \star \Rightarrow \star' $.  
\end{proof}

\subsubsection{Biproducts in 2-categories or 2-Biproducts}
We are now in a good position to define biproducts in 2-categories. Having defined constant 2-morphisms, we can recover algebraic and limit-form definitions of biproducts in Hom-categories. Before defining the algebraic definition of 2-biproducts, we focus on the definition based on 2-products/coproducts. We substitute categories with 2-categories and limits with 2-limits. One can consider four possible limits in 2-categories: strict, weak, lax, and oplax \cite{borceux1994handbook}. We restrict our attention to weak 2-limits; the strict version is obtained by letting all weakening 2-isomorphisms be identities. The following definition by Borceux is given in Categorical Algebra volume I~\cite{borceux1994handbook}.
\begin{defi}
	In a 2-category, a \textit{weak 2-product} of a pair of objects $A, B$ is an object $A \times B$ equipped with 1-morphism projections $(p_A: A\times B \longrightarrow A, p_B: A\times B \longrightarrow B)$ such that: 
	\begin{itemize}
		\item for every cone $(X, f: X \longrightarrow A,  g: X \longrightarrow B)$, there exist a 1-morphism $b: X \longrightarrow A \times B$ and 2-isomorphisms $\{\xi\}$ such that $(\xi_A: p_A   b \Longrightarrow f , \xi_B: p_B   b \Longrightarrow g)$ (the red cone in Figure \ref{fig:product}).
		\item Moreover, for any other cone $(X, f': X \longrightarrow A, g': X \longrightarrow B)$ with a corresponding 1-morphism $b': X \longrightarrow A\times B$ and 2-isomorphisms $(\xi'_A: p_A   b' \Longrightarrow f', \xi'_B: p_B   b' \Longrightarrow g')$ (the blue cone in Figure \ref{fig:product}) and given 2-morphisms $(\Sigma_A: f \Longrightarrow f', \Sigma_B: g \Longrightarrow g')$, 
	\end{itemize}
	there exists a unique 2-morphism $\gamma: b \Longrightarrow b'$ which satisfies the following condition:
	\begin{equation} \label{eq:product} (p_A   \gamma)  = ( {\xi'}_A)^{-1}\odot \Sigma_A \odot (\xi_A ) \end{equation}
	\begin{figure}[h]
		\centering
		\begin{tikzpicture}
			\node [] (0) at (-2, -1) {$A$};
			\node [] (1) at (2, -1) {$B$};
			\node [] (2) at (0, 1) {$A \times B$};
			\node [] (3) at (0, 3) {$X$};
			\node [] (4) at (1, 1.8) {};
			\node [] (5) at (-1, 1.8) {};
			\node [] (6) at (-1.5, 1.5) {};
			\node [] (7) at (-2, 2) {};
			\node [] (8) at (1.5, 1.75) {};
			\node [] (9) at (2, 2) {};
			\node [] (10) at (-0.25, 2) {};
			\node [] (11) at (0.25, 2) {};
			\node [red] (12) at (-0.8, 2.6) {$f$};
			\node [blue] (13) at (-1.9, 2.6) {$f'$};
			\node [red] (14) at (0.8,  2.6) {$g$};
			\node [blue] (15) at (1.95,2.6) {$g'$};
			\node [] (16) at (0.5, 1) {};
			\node [] (17) at (-0.5, 1) {};
			\node [] (18) at (1, 0.5) {};
			\node [] (19) at (-1, 0.5) {};
			\node [] (20) at (2.5, 1) {};
			\node [] (21) at (-2.5, 1) {};
			
			\draw [->, red, dashed]  (3)  to [ bend right=20]  (0);
			\draw [->, blue]  (3)  to [bend right=70] node[left] {} (0);
			\draw [->] (2) to [bend right=20]  node[below right] {$p_A$} (0);
			\draw [->, red, dashed]  (3)  to [bend left=20] (1);
			\draw [->, blue]  (3)  to [bend left=70] (1);
			\draw [->] (2) to [bend left=20] node[below left] {$p_B$}(1);
			\draw [blue, ->] (3) to [bend left=30] node[right] {$b'$}(2);
			\draw[red, dashed, ->] (3) to [bend right=30] node[left] {$ b$}(2); 
			\draw[double,thin,-> ](6) to node[below, midway, sloped] {$\Sigma_A$} (7);
			\draw[double,thin,->] (8) to node[below, midway, sloped] {$\Sigma_B$}(9);
			\draw[double,thin,->] (10) to node[above] {$ \gamma$}(11);
			\draw[red, dashed, double,thin,->, bend right=20](16) to node[below, midway, sloped] {$ \xi_B$} (4);
			\draw[red, dashed, double,thin,->, bend left=20] (17) to node[below, midway, sloped] {$\xi_A$}(5);
			\draw[blue, double,thin,->, bend right=20](18) to node[below right, midway, sloped] {$ \xi'_B$} (20);
			\draw[blue, double,thin,->, bend left=20] (19) to node[below left, midway, sloped] {$\xi'_A$}(21);
		\end{tikzpicture}
		\caption{Weak 2-product in 2-categories.}\label{fig:product}	
	\end{figure}
\end{defi}
\begin{defi}
	In a 2-category, a \textit{strict 2-product} of a pair of objects is a  weak 2-product whose weakening 2-isomorphisms $\xi_A$ and $\xi_B$ are identities.
\end{defi}
\begin{rem}
Observe that unlike product/coproduct definition, because a 1-morphism from $X$ to $A \times B$ is not unique, corresponding to each of these 1-morphisms, there exists a cone with the same apex $X$.  Also, for weak 2-limits, each of the side triangles commutes up to a 2-isomorphism.  
\end{rem}
\begin{defi}
	A \textit{locally semiadditive 2-category} is a 2-category whose $\hom$-categories are semiadditive(have finite biproducts defined in Section~\ref{sect:biproduct}). 
\end{defi}
\begin{defi}
	A \textit{locally semiadditive and compositionally distributive 2-category} is a locally semiadditive 2-category whose 2-morphisms distribute over addition of 2-morphisms. That is Equations \ref{eq12} and \ref{eq13} hold in a compositionally distributive 2-category:
	\begin{align}
		\label{eq12}  &  \gamma  (\alpha +\beta)= \gamma  \alpha + \gamma   \beta \\
		\label{eq13}	&    \alpha \odot (\beta+ \gamma)= \alpha \odot \beta + \alpha \odot \gamma
	\end{align}
\end{defi}
\begin{rem}
	In a locally semiadditive 2-category, biproducts in Hom-categories are biproducts of 1-morphisms. In the proceeding sections, projections $\pi$ and injections $\nu$ are 2-morphisms indexed by 1-morphisms.
\end{rem}
\begin{proposition}\label{pro-dist}
	In a locally semiadditive and compositionally distributive 2-category, the composition of 1-morphisms is distributive over biproducts of 1-morphisms.
	$$ f   (g \oplus h) \cong f   g \oplus f   h $$
\end{proposition}
\begin{proof} 
To prove this proposition, we need to show 2-morphisms $\alpha$ and $\alpha'$ in Figure~\ref{fig-proposition5} are inverse of each other. From universality of products of 1-morphisms shown in Figure \ref{fig-proposition5}, we have:
	\begin{align}\label{eq:eq7}
		& (f \pi_{h}) \odot \alpha = \pi_{f  h},  &&
		(f   \pi_{g}) \odot \alpha = \pi_{f  g}
	\end{align}
	and universality of coproducts of 1-morphisms results in: 
	\begin{align}\label{eq:eq8}
		&\alpha' \odot (f  \nu_{h}) = \nu_{f  h}, &&
		\alpha' \odot (f   \nu_g)=\nu_{f  g} 
	\end{align}
	Composing Equations \ref{eq:eq7} and \ref{eq:eq8}, we have:
	\begin{align}
		& \alpha'\odot (f   \nu_h) \odot (f    \pi_h) \odot \alpha = \nu_{f  h} \odot \pi_{f  h}\label{eq2} \\
		& \alpha'\odot (f    \nu_g) \odot (f    \pi_g) \odot \alpha = \nu_{f  g} \odot \pi_{f  g} \label{eq3}
	\end{align}
	Considering distributivity conditions \ref{eq12} and \ref{eq13} and adding two sides of Equations \ref{eq2} and \ref{eq3}, we obtain:
	$$ \alpha'\odot (f   [(\nu_g \odot \pi_g )+(\nu_h \odot \pi_h )] \odot \alpha = \nu_{f  g} \odot \pi_{f  g}+\nu_{f  h} \odot \pi_{f  h}\\
	$$
	which is:
	$\alpha' \odot (1_f  \circ 1_{g\oplus h}) \odot \alpha = 1_{(f   g \oplus f   h)}$, hence, $\alpha'  \odot  \alpha=1_{f   g \oplus f   h}$. One can also show that $\alpha   \odot \alpha'=1_{f   (g \oplus h)}$ by using the equations above. 
	\begin{figure}[h]
		\centering
		\begin{tikzpicture}[scale=0.8]
			\node (0) at (0, 0) []{$f   g$};
			\node (1) at (3, 3) []{$f   (g \oplus h)$};
			\node (2) at (6, 0) []{$f   h$};
			\node (3) at (3, 5)[]{\small{$f   g \oplus f   h$}};
			\node (4) at (0.9, 2.1) []{};
			\node (5) at (1.9, 0.6) {};
			\node (6) at (5.1, 2) {};
			\node (7) at (4.2, 0.6) {};
			
			\draw[->, bend right=15, double](0) to node[sloped, below]{\small{$f   \nu_g$}} (1);
			\draw[->, bend right=15, double, color=red](1) to node[sloped, above]{\small{  $f   \pi_g$}} (0);
			\draw[->, bend left=15, double, color=red](1) to node[sloped, above ]{\small{$f   \pi_h$}} (2);
			\draw[->, bend left=15, double](2) to node[sloped, below]{\small{$f   \nu_h$}} (1);
			\draw[->, bend left=20, double, color=red](3) to node[right]{$\alpha$}(1);
			\draw[->, bend left=20, double](1) to node[left]{$\alpha'$}(3);
			\draw[->, bend left=30, double](0) to node[sloped, above]{\small{$\nu_{f  g}$}}(3);
			\draw[->, bend right=30, double](2) to node[sloped, above]{\small{$\nu_{f  h}$}}(3);
			\draw[->, bend right=60, double, color=red](3) to node[sloped, above]{\small{$\pi_{f  g}$}}(0);
			\draw[->, bend left=60, double,color=red](3) to node[sloped, above]{\small{$\pi_{f  h}$}}(2);
		\end{tikzpicture}
		\caption{Proof of Proposition \ref{pro-dist}.}\label{fig-proposition5}
	\end{figure}
\end{proof}
\begin{proposition}
	In a locally semiadditive and compositionally distributive 2-category, zero 1-morphisms are constant.
\end{proposition}
\begin{proof}
	Consider the proposition above, for every $g$ we have $g   (f \oplus 0) \cong g   f \oplus g   0$. On the other hand, $f \oplus 0 \cong f$ which shows $g   0 \cong 0$ because zero 1-morphisms are unique up to 2-isomorphisms. 
\end{proof}
We have developed enough structures to be able to introduce the definition of 2-biproducts:   
\begin{defi}\label{definition-bipro}
	In a locally semiadditive and compositionally distributive 2-category, a \textit{weak 2-biproduct} of a pair of objects $(A, B)$  is a tuple
	$$(A\boxplus B, p_A, p_B, i_A, i_B, \theta_A, \theta_B, \theta_{AB}, \theta_{BA}, \theta_{P})$$ such that:
	\begin{itemize}
		\item 
		\textit{1-Morphism projections and injections:}
		\begin{align*}
			& p_{A}: A \boxplus B \longrightarrow A,  && p_{B}: A \boxplus B \longrightarrow B 
			\\
			& i_{A}: A \longrightarrow A \boxplus B,  && i_{B}: B \longrightarrow A \boxplus B 
		\end{align*}
		\item 
		\textit{Weakening 2-isomorphisms:}
		\begin{align*}
			& \theta_{A}: p_A   i_A \Rightarrow \id_A, && \theta_{B}: p_B   i_B \Rightarrow \id_B, 
			\\
			& \theta_{BA}: p_B   i_A \Rightarrow 0_{B, A}, && \theta_{AB}: p_A   i_B \Rightarrow 0_{A, B}, \\
		&\theta_{P}: i_A   p_A \oplus i_B   p_B \Rightarrow \id_{A \boxplus B} 
        \end{align*}
		\item \textit{Conditions for 2-biproducts:}\\
		\begin{align}\label{conditions-of-2-biproducts}
			& p_A   \theta_P   i_A= \begin{bmatrix}
				(p_A   i_A)   \theta_A & 
				0  \\
				0 & 0 
			\end{bmatrix}, \hspace{1cm}
			&  p_B   \theta_P   i_B= \begin{bmatrix}
				0 & 0 \\
				0 &
				(p_B   i_B)   \theta_B 
			\end{bmatrix}
		\end{align}
	\end{itemize}
\end{defi}
\begin{defi}
	In a locally semiadditive and compositionally distributive 2-category, a \textit{strict 2-biproduct} is a weak 2-biproduct whose weakening 2-isomorphisms $\{ \theta \}$ are 2-identities. 
\end{defi}
\begin{defi}
	A \textit{semiadditive 2-category} is a locally semiadditive and compositionally distributive 2-category with weak binary 2-biproducts and a zero object. 
\end{defi}
To prove the consistency of this definition with weak 2-products, we need the following lemmas.
\begin{lemma}\label{lemma33}
In a 2-category, if $p_A   i_A: A \longrightarrow A$ and  $\theta_A: p_A   i_A \Rightarrow \id_A$, then $(p_A   i_A )  \theta_A=\theta_A (p_A   i_A ) $. 
\end{lemma}
\begin{proof}
It is straightforward if you follow the composition diagram.
\begin{center}
\begin{tikzpicture}
    \node (0) at (0, 0) {$A$};
    \node (1) at (2, 0) {$A$}; 
    \node (2) at (4, 0) {$A$}; 
    \node (3) at (6, 0) {$A$};
    \node (6) at (3, 0.5){};
    \node (7) at (3, -0.5){};

    \draw[->] (0) to node[above]{$p_Ai_A$}(1);
    \draw[->, bend left=45] (1) to node[above]{$p_Ai_A$} (2);
    \draw[->, bend right=45] (1) to node[below]{$id_A$} (2); 
    \draw[->, double] (6) to node[right]{$\theta_A$} (7);
    \draw[->] (2) to node[above]{$p_Ai_A$}(3);
\end{tikzpicture}    
\end{center}

\end{proof} 
\begin{lemma}\label{lemma2}
	In a locally semiadditive 2-category,  $\theta_{BA}: p_B   i_A \Rightarrow 0_{B, A}$ and $\theta_{AB}: p_A   i_B \Rightarrow 0_{A, B}$ are zero 2-morphisms. 
\end{lemma}
\begin{proof}
	Because $0_{A, B}$ is a zero object in Hom-category $Hom(A, B)$ and because each set of 2-morphisms should have a zero 2-morphism and $\theta_{A B}$ is unique, it is indeed a zero 2-morphism. 
\end{proof}
			
\begin{lemma}\label{lemma-34}
In a 2-category, for a unique 2-morphism $\gamma: h \Rightarrow h'$, if there exists a 2-isomorphism $\theta_P: l \Rightarrow id_B$, then $\gamma$ is transformed by the following formula: 
	\begin{equation}
		(\theta_{P}   h') \odot (l   \gamma) \odot (\theta_P^{-1}   h)=\gamma
	\end{equation}
\end{lemma}
\begin{proof} Follows from the Figure~\ref{fig:lemma5}.
	\begin{figure}[h!]
    \centering
		\begin{tikzpicture}[scale=0.7]
			\node (0) at (0, 0) []{$A$};
			\node (1) at (3, 0) []{$B$};
			\node (2) at (4, 0) []{$=$};
			\node (3) at (5, 0) []{$A$};
			\node (4) at (10, 0) []{$B$};
			\node (5) at (6, 1) []{};
			\node (6) at (6, 0) []{};
			\node (7) at (6, -1) []{};
			\node (8) at (6, -2) []{};
			\node (9) at (6, -3) []{};
			\node[rotate=-90](10) at (1.5, 0) []{$\Rightarrow$};
			\node (11) at (1.2, 0) {$\gamma$};
			\node[rotate=-90] (12) at (7.5, 1.2) []{$\Rightarrow$};
			\node (13) at (8.4, 1.2) []{$\theta_P^{-1}h$};
			\node (14) at (5.9, 0.7)[]{$l  h$};
			\node (15) at (5.9, -0.7)[]{$l  h'$};
			\node[rotate=-90] (16) at (7.5, 0)[]{$\Rightarrow$};
			\node (17) at (8.4, 0)[]{$l\gamma$};
			\node[rotate=-90](18) at (7.5, -1.1)[]{$\Rightarrow$};
			\node (19) at (8.4, -1.2)[]{$\theta_Ph'$};
			
			\draw[->, bend left=60] (0) to node[above]{$h$}  (1);
			\draw[->, bend right=60] (0) to node[below]{$h'$} (1);
			\draw[->, bend left=110] (3) to node[above]{$h$} (4);
			\draw[->, bend left=20] (3) to node[below]{} (4);
			\draw[->, bend right=20] (3) to node[below]{} (4);
			\draw[->, bend right=110] (3) to node[below]{$ h'$} (4);
		\end{tikzpicture}
        \caption{Proof of Lemma~\ref{lemma-34}.}\label{fig:lemma5}
	\end{figure}
\end{proof}
\begin{lemma}\label{lemma35}
	In a locally semiadditive 2-category, considering the definition of 2-biproducts and $r= i_A   p_A \oplus i_B   p_B $, the 2-morphisms  $\Sigma_A: p_A   r \Rightarrow p_A $ and $\Sigma_B: p_B   r \Rightarrow p_B$ are given by the following row matrices:
	\begin{align}
		&&\Sigma_A=\begin{bmatrix}
			\theta_A   p_A &
			0  
		\end{bmatrix}, && \Sigma_B=\begin{bmatrix}
			0 &
			\theta_B   p_B
		\end{bmatrix} 
	\end{align}
\end{lemma}
\begin{proof}
	Given Diagram \ref{dig-lemma}, we have:
	$\lambda=\begin{bmatrix}
		\theta_A   p_A & 0 \\
		0 &  \theta_{AB}   p_B 
	\end{bmatrix} $. To obtain, $\Sigma_A$, we must calculate $\pi_1 \odot \lambda$, which is
	$$
	\pi_1 \odot \lambda = \begin{bmatrix}
		1 & 0
	\end{bmatrix} \odot \begin{bmatrix}
	\theta_A   p_A & 0 \\
	0 &  \theta_{AB}   p_B 
\end{bmatrix}   = \begin{bmatrix}
		\theta_A p_A & 0
	\end{bmatrix} 
	$$
	One can similarly show that for $\pi_2 \odot \lambda^\prime= \Sigma_B$, if $\lambda^\prime = \begin{bmatrix}
		\theta_{B, A}   p_A & 0 \\
		0 &  \theta_{B}   p_B 
	\end{bmatrix} $. 
	\begin{figure}[h!]
		\centering
		\begin{tikzpicture}[scale=0.65]
			\node (1) at (0, 0) []{$p_A$};
			\node (2) at (2, 2) []{$p_A   i_A   p_A$};
			\node (3) at (4, 4) []{$p_A  r$};
			\node (4) at (8, 0) []{$0_{A, B}$};
			\node (5) at (6, 2) []{$p_A   i_B   p_B$};
			\node (6) at (4, 6) []{$p_A \oplus 0_{A}$};
			
			\draw[ ->, double] (2) to node[below right] {\small{$\theta_A   p_A$}}(1);
			\draw[ ->, bend right=20, double ] (2) to node[above, midway]{\small{$\nu_1$}~~~}(3);
			\draw[ ->, bend right=30, double ] (3) to node[left]{\small{$\pi_1$}}(2);
			\draw[ ->, double] (5) to node[left] {\small{$ \theta_{AB}   p_B$}}(4);
			\draw[ ->, bend left=20, double ] (5) to node[above, midway]{\small{~~~$\nu_2$}}(3);
			\draw[ ->, bend left=30, double ] (3) to node[right]{\small{$\pi_2$}}(5);
			\draw[->, double](3) to node[right]{$\lambda$}(6) ;
			\draw[ ->, bend left=40, double](1) to node[left]{$\nu_1$}(6);
			\draw[ ->, bend right=40, double](4) to node[right]{$\nu_2$}(6);
		\end{tikzpicture}
		\caption{The 2-morphism $\lambda$ in the figure is $\lambda=\nu_1 \odot \theta_A   p_A \odot \pi_1 + \nu_2 \odot \theta_{AB} p_B \odot \pi_2$.}\label{dig-lemma} 
	\end{figure} 
\end{proof}
\begin{lemma}\label{lemma-35}
	In a semiadditive 2-category, for every 1-morphism written as $h=i_A   f \oplus i_B   g $ and given a 2-isomorphism $\theta_P:  i_A p_A \oplus i_B p_B \Rightarrow \id_{A\boxplus B}$ we have:
	\begin{equation}
		\theta_P   h =\begin{bmatrix}
			i_A   \theta_A   f & 0\\
			0 & i_B   \theta_B   g
		\end{bmatrix} 
	\end{equation}
\end{lemma}
\begin{proof} The essence of the proof is in the Figure~\ref{fig:proof-lemma-7}, in which $\alpha = \theta_Ph \odot \nu_1$ and $\beta = \theta_ph \odot \nu_2$. Because $\alpha$ and $\beta$ determine $\theta_Ph$, we only need to find them.
\begin{figure}[h!]
    \centering
	    \begin{tikzpicture}[scale=0.5]
			\node (2) at (1, 1.5) []{$i_A   p_A   h$};
			\node (3) at (4, 5) []{\small{$(i_A   p_A \oplus i_B   p_B)  h $}};
			\node (5) at (7, 1.5) []{$i_B   p_B  h$};
			\node (6) at (4, 8) []{$h$};
			
			\draw[ -> , double] (2) to node[above, sloped]{\small{$\nu_1$}}(3);
			\draw[ ->, double] (5) to node[above, sloped]{\small{$\nu_2$}}(3);
			\draw[->, double](3) to node[right]{\small{$\theta_P   h$}}(6) ;
			\draw[ ->, bend left=50, double](2) to node[left]{\small{$\alpha$}}(6);
			\draw[ ->, bend right=50, double](5) to node[right]{\small{$\beta$}}(6);
		\end{tikzpicture}
\caption{Proof of Lemma 7}\label{fig:proof-lemma-7}
\end{figure}
		
We can determine $\alpha$ by writing the explicit forms of $i_Ap_Ah$ and $i_Bp_Bh$ to find the 2-morphism from $i_Ap_Ah$ to $h$. 
	\begin{equation*}
		\begin{tikzpicture}
			\node (0) at (0, 0) {$i_Ap_Ah = i_Ap_Ai_Af \oplus i_Ap_Ai_Bg$};
			\node (1) at (4, 0){$i_Ap_Ai_Af$};
			\node (2) at (7, 0){$i_Af$};
			\node (3) at (9, 0){$h$};
			\draw[->, double] (0) to node[above]{$\pi_1$} (1);
			\draw[->, double] (1) to node[above]{$i_A\theta_Af$} (2);
			\draw[->, double] (2) to node[above]{$\nu_1$} (3);
		\end{tikzpicture}
	\end{equation*}
	Starting from $i_Bp_Bh$ and following the same procedure, we can find $\beta$. 
	\begin{align*}
		& \alpha = \nu_1 \odot i_A\theta_{A}f \odot \pi_1, & \beta = \nu_2 \odot i_B \theta_B g \odot \pi_2
	\end{align*}
\end{proof}
We now present the main theorem of this paper. Similar to the categorical case, which was discussed in References \cite{murfet2006abelian, borceux1994handbook},  we show the consistency of the algebraic definition of weak 2-biproducts with weak 2-products/2-coproducts. The theorem essentially proves both of our definitions are equivalent. 
\begin{theorem}\label{theorem-main}
	In a locally semiadditive and compositionally distributive 2-category, the following conditions for a pair of objects $A$ and $B$ are equivalent:
	\begin{enumerate}
		\item the weak 2-product $(P, p_A, p_B)$ of $A, B$ exists. 
		\item the weak 2-coproduct $(P, i_A, i_B) $ of $A, B$ exists.
		\item the weak  2-biproduct $(P, p_A, p_B, i_A, i_B, \theta_A, \theta_B, \theta_{AB}, \theta_{BA}, \theta_{P})$ of $A, B$ exists.
	\end{enumerate}	
\end{theorem}
\begin{proof}  1 $\Longrightarrow$ 3: Assuming a pair of objects $A, B$ has a weak 2-product $P$, we want to show $P$ is also the weak 2-biproduct.
	To this end, check the universal property of weak 2-products for $(A, \id_A: A\longrightarrow A, 0_{B, A}: A \longrightarrow B, i_A: A \longrightarrow P)$ and $(B, \id_B: B\longrightarrow B, 0_{A, B}: B \longrightarrow A, i_B: B \longrightarrow P)$. From the definition of 2-products, we know that there exist two 2-morphisms $(\gamma_A: p_A   i_A \Longrightarrow \id_A , \gamma_B: p_B   i_B \Longrightarrow \id_B )$. We let $\theta_{A}:=\gamma_{A}$ and $\theta_B:=\gamma_B$. To find $\theta_p$, we check the universality condition for $P$  and two 2-cones:
	\begin{enumerate}
		\item $(P, p_A   l: P \longrightarrow A, p_B  l: P \longrightarrow B, l: P \longrightarrow P), \hspace{0.5cm} l:=i_A   p_A \oplus i_B   p_B$
		\item $(P, p_A: P \longrightarrow A, p_B: P \longrightarrow B, \id_P: P \longrightarrow P)$
	\end{enumerate} 
	By Lemma \ref{lemma35}, 2-morphisms between these two cones are as below:
	\begin{align*}
		&& \Sigma_A=\begin{bmatrix}
			\gamma_A   p_A &
			0  
		\end{bmatrix}: p_A   i_A   p_A\oplus p_A   i_B   p_B \Rightarrow p_A,\\
		&& \Sigma_B=\begin{bmatrix}
			0 & \gamma_B   p_B
		\end{bmatrix}: p_B   i_A   p_A\oplus p_B   i_B   p_B \Rightarrow p_B
	\end{align*}
	Therefore, 
	$\theta_P =\begin{bmatrix}
		i_A   \Sigma_A \\
		i_B   \Sigma_B
	\end{bmatrix} = \begin{bmatrix}
		i_A   \gamma_A   p_A & 0 \\
		0 & i_B   \gamma_B   p_B 
	\end{bmatrix} $.
	\\
Considering Lemma \ref{lemma33}, also the point that $(1_{p_A   i_A} )^2=1_{p_A   i_A}$, we can show $\theta_P$ satisfies necessary conditions. 
	$$
	p_A   \theta_P   i_A=
	\begin{bmatrix}
		(p_A   i_A)   \theta_A    (p_A   i_A) & 0 \\
		0 & 0
	\end{bmatrix}
	= \begin{bmatrix}
		(p_A   i_A)   \theta_A   & 0 \\
		0 & 0
	\end{bmatrix}.$$
	3 $\Longrightarrow$ 1:
	Now if we have the algebraic definition of 2-biproduct for a pair of objects, we want to demonstrate $P$ is also a 2-product. Because we already have projections, it is enough to show it satisfies the universality condition. To this end, check the universality condition for an arbitrary cone $(X, f: X\longrightarrow A, g: X \longrightarrow B)$ and let a 1-morphism from $X$ to $P$ be $h=i_A  f \oplus i_B   g$:
	\begin{align*}
		&&\xi_A: p_A   h= p_A   (i_A  f \oplus i_B  g) \Rightarrow f, && \xi_A = \theta_A   f \odot \pi_1
	\end{align*}
	Similarly, we obtain $\xi_B = \theta_B   g \odot \pi_2$. Given another cone $(X, f': X\longrightarrow A, g': X \longrightarrow B)$, Letting $h'=i_A   f' \oplus i_{B}  g'$, we have $\xi'_A = \theta_A   f'\odot \pi_1$ and $\xi'_B = \theta_B   g'\odot \pi_2$. If there exist 2-morphisms $(\Sigma_A: f\Longrightarrow f', \Sigma_B: g\Longrightarrow g')$, we should show that there is a unique 2-morphism $\gamma: i_A  f\oplus i_B  g \Longrightarrow i_A  f'\oplus i_B  g'$ such that
	$(p_A   \gamma)=(\xi'_A)^{-1}\odot \Sigma_A \odot \xi_A$. Suppose $\gamma$ in the explicit form is $
	\gamma = \nu_1 \odot i_A \Sigma_A \odot \pi_1 + \nu_2 \odot i_B \Sigma_B \odot \pi_2
	$ and in the matrix form is as follows:
	$$ \gamma=
	\begin{bmatrix}
		i_A\Sigma_A & 0 \\
		0 & i_B \Sigma_B
	\end{bmatrix}
	$$
	Multiplying $\gamma$ by $p_A$, we have:
	\[
	p_A \gamma = \nu_1 \odot p_A i_A \Sigma_A \odot \pi_1 + \nu_2 \odot p_A i_B \Sigma_B \odot \pi_2 
	\]
	Note that $p_A   i_A   \Sigma_A = (\theta^{-1}_A   f^\prime) \odot \Sigma_A \odot (\theta_A   f) $ and $p_A i_B \cong 0_{A, B}$, so:
	\[
	p_A \gamma = \nu_1 \odot (\theta^{-1}_A   f^\prime) \odot \Sigma_A \odot (\theta_A   f) \odot \pi_1 
	\]
	Additionally, the inverse of $\xi_A = (\theta_A f) \odot \pi_1$ is $(\xi_A)^{-1} = \nu_1 \odot (\theta_A^{-1} f )$. Substitution of these terms results in the desired condition: 
	\[
	p_A \gamma = \xi_A^{-1} \odot \Sigma_A \odot \xi_A 
	\]
	\textbf{\underline{Uniqueness of $\gamma$}}: Assume that there exists another 2-morphism $\gamma': i_A  f \oplus i_B   g \Rightarrow i_A   f' \oplus i_B   g' $ such that it satisfies all necessary conditions. Using lemma \ref{lemma-34} and lemma \ref{lemma-35}, we prove $\gamma'$ is equivalent to $\gamma$:
	\begin{align*}
		& \gamma'=(\id_p)  \gamma'\overset{\text{lemma~} \ref{lemma-34}}{=}(\theta_P h')\odot [(i_A  p_A \oplus i_B  p_B) \gamma'] \odot ({\theta_P}^{-1} h) \xRightarrow[]{\text{Equation \ref{eq:product}, lemma \ref{lemma-35}}}\\
		& \gamma'= \begin{bmatrix}
			i_A   (\theta_A   f') & 0\\
			0&i_B   (\theta_B   g')
		\end{bmatrix} 
		\odot \begin{bmatrix}
			i_A   [(\theta_A^{-1}  f') \odot \Sigma_A \odot (\theta_A  f)]  &0   \\
			0&i_B   [(\theta_B^{-1} g') \odot \Sigma_B \odot (\theta_B   g)]
		\end{bmatrix} \\
		&	\odot \begin{bmatrix}
			i_A   (\theta_A^{-1}   f) & 0\\
			0& i_B   (\theta_B^{-1}   g)
		\end{bmatrix} =\begin{bmatrix}
			i_A \Sigma_A & 0\\
			0 & i_B \Sigma_B
		\end{bmatrix} = \gamma.
	\end{align*}
	In the second line, the first and third matrices are written by using Lemma \ref{lemma-35}, and the middle matrix is obtained by condition of weak 2-products Equation \ref{eq:product}. 
	
	Note that $\theta_P   h = \nu_1 \odot (i_A  \theta_A  f) \odot \pi_1 + \nu_2 \odot (i_B   \theta_B  g) \odot \pi_2$.
	If one lets $X=P$, $f=\id_A$, and $g=0_{B, P}$; this yields $h=i_A$ and $\xi_A= \theta_A   \id_A$, hence, $\theta_P   h = \theta_P   i_A$. Multiplying $\theta_P h$ with $p_A$, we obtain one of the conditions of 2-biproducts.
	$$p_A   \theta_P   i_A = \begin{bmatrix}
		(p_A   i_A)   \theta_A  & 0\\
		0 & 0
	\end{bmatrix}$$
\end{proof}
As mentioned earlier, the categorification of the canonical definition of biproducts is achievable in a rather simple way by only using the notion of weak 2-limits in 2-categories. In the following, we first state the definition obtained from the limit notion and examine the compatibility of this definition with our algebraic version. 
\begin{defi}\label{def2.2}
	In a locally semiadditive and compositionally distributive ~~2-category, a \textit{canonical 1-morphism} between a 2-coproduct of a pair of objects $A, B$, i.e. $A \sqcup B$ and a 2-product $A \times B$ is a 1-morphism which satisfies $\theta_{k, j}: p_k   r   i_j \Rightarrow \delta_{k, j} id_j$, if $\theta_{k, j}$ are 2-isomorphisms. 
	$$ A_j \xrightarrow{i_j} A_1 \sqcup A_2 \xrightarrow{r} A_1 \times A_2 \xrightarrow{p_k} A_k, \hspace{0.5cm}\text{if} \hspace{0.5cm} j, k \in \{1, 2\} $$
\end{defi}
\begin{rem}
	A canonical 1-morphism $r$ is unique up to a unique 2-isomorphism. That is, for every pair of such 1-morphisms, there exists a unique 2-isomorphism $ \gamma: r \Rightarrow r^\prime$ such that $p_k   \gamma   i_j = {(\theta^\prime_{k, j})}^{-1} \odot \theta_{k, j}$.
\end{rem}
\begin{rem}
	2-Isomorphisms $\theta_{i, j}$ for $i \neq j$ are zero 2-morphisms, according to Lemma \ref{lemma2}. 
\end{rem} 
\begin{defi}\label{def1.1}
	In a locally semiadditive and compositionally distributive ~~2-category, a pair of objects $A$ and $B$ has a \textit{weak 2-biproduct} if the canonical 1-morphism $r$ is an equivalence and satisfies the following conditions:
	\begin{align*}
		& A_1 \sqcup A_2 \xrightarrow{r} A_1 \times A_2, \hspace{1cm} & A_1 \sqcup A_2 \xleftarrow{r^\prime} A_1 \times A_2
		\\
		& \xi_{A \times B}: r r^\prime \Rightarrow id_{A \times B}, \hspace{1cm} &\xi_{A \sqcup B}: id_{A \sqcup B} \Rightarrow r^\prime r \\
		& 
		(r^\prime \xi_{A \times B})\odot (\xi_{A \sqcup B} r^\prime) = 1_{r^\prime}, & \hspace{1cm} (\xi_{A \times B} r) \odot (r \xi_{A \sqcup B}) = 1_r
	\end{align*}
\end{defi}
To show the consistency of this definition with the algebraic version, we first need to prove 1-morphism projections are weakly monic.
\begin{defi}
	In a 2-category, a 1-morphism $f:B \longrightarrow C$ is \textit{weakly monic} if for every pair of 1-morphisms $g, h: A \longrightarrow B$, if $fg$ is isomorphic to $fh$, i.e. $fg \cong fh$. Then $g$ and $h$ are isomorphic $g \cong h$. 
\end{defi}
\begin{lemma}\label{lemma4}
	In a 2-category with binary 2-products, 1-morphism projections are monomorphisms.  
\end{lemma}
\begin{proof}	
	To prove $p_A$ is mono, we need to show that if $p_A   b \overset{\Sigma_A}{\cong} p_A   b^\prime$ there exists a 2-isomorphism $\gamma: b \Rightarrow b^\prime$. With 1-morphisms $b$ and $b'$, we can make two 2-cones $(X, p_Ab, p_Bb, b)$ and $(X, p_Ab', p_Bb', b')$. Due to our assumption, there exists 2-isomorphisms $\Sigma_A: p_A   b \Rightarrow  p_A   b'$ and $\Sigma_B: p_B   b \Rightarrow  p_B   b'$, hence, a unique 2-isomorphism $\gamma: b \Rightarrow b^\prime$ such that $p_A   \gamma = \Sigma_A$ and $p_B   \gamma = \Sigma_B $. 
\end{proof}
\begin{figure}[h]
\centering
\begin{tikzpicture}[scale=0.85]
		\node [] (0) at (-2, -1) {$A$};
		\node [] (1) at (2, -1) {$B$};
		\node [] (2) at (0, 1) {$A \times B$};
		\node [] (3) at (0, 3) {$X$};
		\node [] (4) at (1, 1.8) {};
		\node [] (5) at (-1, 1.8) {};
		\node [] (6) at (-1.5, 1.75) {};
		\node [] (7) at (-2, 2) {};
		\node [] (8) at (1.5, 1.75) {};
		\node [] (9) at (2, 2) {};
		\node [] (10) at (-0.25, 2) {};
		\node [] (11) at (0.25, 2) {};
		\node [red] (12) at (-1, 2.5) {$p_Ab$};
		\node [blue] (13) at (-2.3, 2.5) {$p_Ab'$};
		\node [red] (14) at (1.2,  2.5) {$p_Bb$};
		\node [blue] (15) at (2.4, 2.5) {$p_Bb'$};
		\node [] (16) at (0.5, 1) {};
		\node [] (17) at (-0.5, 1) {};
		\node [] (18) at (1, 0.5) {};
		\node [] (19) at (-1, 0.5) {};
		\node [] (20) at (2.5, 1) {};
		\node [] (21) at (-2.5, 1) {};
		
		\draw [->, red, dashed]  (3)  to [ bend right=20]  (0);
		\draw [->, blue]  (3)  to [bend right=70] node[left] {} (0);
		\draw [->] (2) to [bend right=20]  node[below right] {$p_A$} (0);
		
		\draw [->, red, dashed]  (3)  to [bend left=20] (1);
		\draw [->, blue]  (3)  to [bend left=70] (1);
		\draw [->] (2) to [bend left=20] node[below left] {$p_B$}(1);
		
		\draw [blue, ->] (3) to [bend left=30] node[right] {$b'$}(2);
		\draw[red, ->, dashed] (3) to [bend right=30] node[left] {$ b$}(2); 
		
		\draw[double,thin,-> ](6) to node[below] {$\Sigma_A$} (7);
		\draw[double,thin,->] (8) to node[below] {$\Sigma_B$}(9);
		\draw[double,thin,->] (10) to node[above] {$ \gamma$}(11);
	\end{tikzpicture} 
\caption{Product in 2-categories.}
\end{figure}
\begin{proposition}
	In a locally semiadditive and compositionally distributive 2-category, a canonical 1-morphism $r$, between a 2-coproduct $A\sqcup B$ and a 2-product $A \times B$ of a pair of objects is an equivalence if and only if projections and injections satisfy the conditions of Definition ~\ref{definition-bipro}. 
\end{proposition}
\begin{proof}
	Suppose we have the algebraic definition of 2-biproducts, that is $\theta_{\alpha, \beta}: p_\alpha   i_\beta \Rightarrow \delta_{\alpha, \beta} id_\alpha$ and $\theta_p: \oplus_\alpha i_\alpha   p_\alpha \Rightarrow id_{A \boxplus B}$. If $\theta_{A, A} = \theta_A$ and $\theta_{B, B}=\theta_B$. To prove a canonical 1-morphism $r$, which satisfies $\xi_{\alpha, \beta}: p_\alpha   r   i_\beta \Rightarrow \delta_{\alpha, \beta}$, has an inverse, we first need to show that $A \boxplus B$ is indeed a 2-product and a 2-coproduct, i.e. it satisfies the universal property, which is exactly the first part of Theorem \ref{theorem-main}.  Therefore, it remains to show that $r$ is an equivalence. Let the inverse be $r^\prime = i_A   p_A \oplus i_B   p_B$. Using $\theta_{\alpha, \beta}: p_\alpha   i_\beta \Rightarrow \delta_{\alpha, \beta} id_\alpha$, we have:
	$$	p_A   r   r^{\prime}= p_A   r   (i_A   p_A \oplus i_B   p_B) \xRightarrow{\lambda} p_A \oplus 0 \xRightarrow{\pi_1} p_A
	$$
	which shows $\lambda = \begin{bmatrix}
		\xi_{A, A} & 0 \\
		0 & \xi_{A, B}
	\end{bmatrix}$. Since $p_A$ is monic, there exists a 2-isomorphism $\xi_{A \times B}: r   r^{\prime} \Rightarrow id_{A \times B}$ such that $p_A   \xi_{A \times B} = \pi_1  \odot \lambda$. One can also show that $r^\prime   r \Rightarrow id_{A \sqcup B}$ by checking $r^\prime   r   i_A$ and using epiciticy of $i_A$. 
	
	$\Leftarrow$ Now given $r$ and $r'$ which satisfy $\xi_{\alpha, \beta}: p_\alpha   r   i_\beta \Rightarrow \delta_{\alpha, \beta}$, $\xi_{A \sqcup B}: r^\prime   r \Rightarrow id_{A \sqcup B}$ and $\xi_{A \times B}: r   r^\prime \Rightarrow id_{A \times B}$, to obtain necessary equations and the condition for 2-biproducts, we definition new injections $i^\prime_A = r   i_A $ and claim the 2-product of two objects with these injections $(A \times B, i^\prime_A, i^\prime_B)$ is a 2-coproduct .i.e. it satisfies the universal property.
\begin{center}
		\begin{tikzpicture}[scale=0.8]
			\node(0) at (0, 0) {$A$};
			\node (1) at (4, 0) {$B$};
			\node (2) at (2, 1.5) {$A \sqcup B$};
			\node (3) at (2, 3) {$A \times B$};
			\node (4) at (2, 4.5) {$X$}; 
			
			\draw[->] (0) to node[above left]{$i_A$} (2);
			\draw[->] (1) to node[above right]{$i_B$} (2);
			\draw[->] (2) to node[right]{$r$}(3);
			\draw[->] (3) to node[right]{$c$}(4);
			\draw[->, bend left=45] (0) to node[left]{$f$}(4);
			\draw[->, bend right=45] (1) to node[right]{$g$}(4);
			\draw[->,  bend left=40] (2) to node[left]{$b$}(4);
		\end{tikzpicture} 
\end{center} 
Because $(A \sqcup B, i_A, i_B)$ is a 2-coproduct, there exists a 1-morphism $b$ such that $\eta_A: b   i_A \Rightarrow f$ and $\eta_B: b   i_B \Rightarrow g$. Let $c= b   r^\prime$, therefore, for $f$ and similarly for $g$: 
	$$
	c i^\prime_A = b r^\prime r i_A \xRightarrow{b\xi_{A \sqcup B}i_A } b~ id_{A \sqcup B}~i_A = bi_A \xRightarrow{\eta_A} f 
	$$
	Moreover, $(A \times B, i^\prime_\alpha,  p_\alpha , \alpha=A, B)$ is a 2-biproduct of $A$ and $B$. That is, we need to find 2-isomorphisms $\theta_{\alpha, \beta}: p_\alpha   i^\prime_\beta \Rightarrow \delta_{\alpha, \beta} id_\alpha$ and $\theta_{p}: i_A^\prime   p_A \oplus i_B^\prime   p_B \Rightarrow id_{A \times B}$ which satisfy appropriate conditions. By condition on $r$, we let 2-isomorphisms $\theta_{\alpha, \beta}:= \xi_{\alpha, \beta}$, 
	Finally, $\theta_{p}: i_A^\prime   p_A \oplus i_B^\prime   p_B \Rightarrow id_{A \times B}$ is the result of the universal property of 2-products and moniticity of projections.
	\begin{align*}
		&&p_A  [i_A^\prime   p_A \oplus i_B^\prime   p_B] = p_A    (r   i_A)   p_A \oplus p_A   (r   i_B)   p_B \xRightarrow{\rho} p_A \oplus 0  \\
		&& p_B  [i_A^\prime   p_A \oplus i_B^\prime   p_B] = p_B    (r   i_A)   p_A \oplus p_B   (r   i_B)   p_B \xRightarrow{\rho^\prime} 0 \oplus p_B 
	\end{align*}	such that $\rho = \begin{bmatrix}
		\xi_{A, A}   p_A& 0 \\
		0 & \xi_{A, B}   p_B
	\end{bmatrix}$ and $\rho^\prime= \begin{bmatrix}
		\xi_{A, B}  p_A & 0 \\
		0 & \xi_{B, B}  p_B
	\end{bmatrix}$. Note that $p_A \theta_P = \rho$ and $p_B \theta_P = \rho'$. So we can check the condition on $\theta_p$, by \\ $p_A   \theta_p   i_A = \begin{bmatrix}
		\xi_{A, A}   p_A   i_A& 0 \\
		0 & \xi_{A, B}   p_B   i_A
	\end{bmatrix}$. 
\end{proof}
So far we have established a definition for semiadditive 2-categories. In next section, we show how we type tensors in this 2-category. 
\subsection{Typing rank four tensors}
We have used matrix notation through this chapter to make the difficult calculations easier and to gain an intuition about the difference between biproducts in categories
and 2-categories. For instance, we have represented the condition on 2-biproducts \ref{conditions-of-2-biproducts} as
$2 \times 2$ matrices. But the conditions on 2-Morphism $\theta_P$ become matrix after applying projections and injections, this observation clearly displays the type of 2-morphisms in semiadditive 2-categories is different from the type of 1-morphisms.

Tensor is an umbrella term for mathematical objects including scalars, vectors, matrices, and so on. Each tensor has a rank specified by their indices, for example, scalars are rank zero tensors, vectors rank one, matrices are rank two tensor, and this list can go on. In linear algebra, they are represented as below, when $e_{\alpha_j}$ is the $j$th basis vector:
\begin{equation}
T = \sum_{\alpha_1, \alpha_2, \alpha_3, ...} t_{\alpha_1, \alpha_2, ...} (e_{\alpha_1} \otimes e_{\alpha_2} \otimes ...)
\end{equation}
Hence, in this notation, 
\begin{align*}
	&\text{Scalar: } n=t & \text{Vector:  } v = \sum_{\alpha_1} t_{\alpha_1} e_{\alpha_1} \\
	&\text{Matrix:  } M = \sum_{\alpha_1, \alpha_2} t_{\alpha_1, \alpha_2} (e_{\alpha_1} \otimes e_{\alpha_2})
\end{align*}
Pay attention to entries of tensors above. Entries of vectors are scalar. If we ﬁx only
one basis, then entries of a matrix become vectors. If we ﬁx both bases of a matrix,
then entries are scalar. Similarly if you fix an index of a rank n tensor, you obtain a rank n-1 tensor, and so force.
\subsection{Category of \vect for typing tensors}
Similar to Section \ref{1-d} that we carried the guiding example of the category of matrices, \mat, 2-Category \vect  introduced by Kapranov and Voevodsky \cite{kapranov19942} inspires us throughout this section.

Kapranov and Voevodsky in their seminal work, introduced
\vect in Definition 5.2, as an example of symmetric monoidal 2-categories. Objects of \vect are natural numbers,
1-morphisms are matrices between natural numbers whose entries are ﬁnite dimensional vector spaces, and 2-morphisms are linear maps between vector spaces. They, however, did not spell out the details and left the rest to the readers. We shall do that here. 

From \vect  perspective, one can observe some properties of semiadditive 2-categories. For instance, in \vect, there are two collections of bases: global bases which are 1-morphism projections and injections $\{p, i\}$ indexed by objects(natural numbers) and local bases, which are the internal bases of each vector space $\{\pi, \nu\}$. Correspondingly, a semiadditive 2-category has a collection of global bases, and for each Hom-category a collection of local bases, which are 2-morphisms projections and injections indexed by 1-morphisms. Global bases are projections $p$ and injections $i$ of 2-biproducts $\boxplus$, and local bases are projections $\pi$ and injections $\nu$ of biproducts $\oplus$ in Hom-categories. Generally, a 2-morphism is a tensor of rank up to four, or in other words, it is a 4-dimensional matrix. Thus, each entry of a matrix has four indices: Latin indices are global and Greek indices are local indices. 
\begin{equation}\label{formula1}
	\theta = \sum_{\alpha, \beta} \sum_{k, l} (\theta_j^k)_\alpha^\beta [(\nu_\alpha \otimes \pi_\beta )\otimes (i_j \otimes p_k )]
\end{equation}
Applying the global projections and injections, one selects an entry which itself is a matrix with entries indexed by
local bases. In a 4-dimensional picture, by applying a 1-morphism projection, one picks a 3-dimensional matrix, a cube. Applying a 1-morphism injection to that cube, we select a square matrix. To pick an entry from this square matrix, one has to apply 2-morphism projections and injections, Figure \ref{figure-global-local-basis}.
\begin{figure}[!h]\label{figure-global-local-basis}
\centering
	\begin{tikzpicture}
		\node [] (0) at (-2, 0) {};
		\node [] (1) at (0, 0) {};
		\node [] (2) at (1, 1) {};
		\node [] (3) at (-1, 1) {};
        \node [] (4) at (0, 2) {};
        \node [] (5) at (1, 3) {};
        \node [] (6) at (-1, 3) {};
        \node [] (7) at (-2, 2) {};
	
		\draw[-] (0) to (1); 
        \draw[-] (1) to (2); 
        \draw[-] (0) to (3); 
        \draw[-] (3) to (6); 
        \draw[-] (3) to (2); 
        \draw[-] (7) to (4); 
        \draw[-] (7) to (0);
        \draw[-] (7) to (6); 
        \draw[-] (1) to (4); 
        \draw[-] (2) to (5); 
        \draw[-] (6) to (5); 
        \draw[-] (4) to (5); 
        
		\node [] (8) at (1, 4) {};
		\node [] (9) at (3, 4) {};
		\node [] (10) at (4, 5) {};
		\node [] (11) at (2, 5) {};
        \node [] (12) at (3, 6) {};
        \node [] (13) at (4, 7) {};
        \node [] (14) at (2, 7) {};
        \node [] (15) at (1, 6) {};
	
		\draw[-] (8) to (9); 
        \draw[-] (9) to (10); 
        \draw[-] (8) to (11); 
        \draw[-] (11) to (14); 
        \draw[-] (11) to (10); 
        \draw[-] (15) to (12); 
        \draw[-] (15) to (8);
        \draw[-] (15) to (14); 
        \draw[-] (9) to (12); 
        \draw[-] (10) to (13); 
        \draw[-] (14) to (13); 
        \draw[-] (12) to (13);

		\draw[-, dashed, red] (0) to (8); 
        \draw[-, dashed, red] (1) to node[below]{$~i_j$}(9) ; 
        \draw[-, dashed, red] (2) to (10); 
        \draw[-, dashed, red] (3) to (11); 
        \draw[-, dashed, red] (4) to (12); 
        \draw[-, dashed, red] (5) to (13); 
        \draw[-, dashed, red] (6) to node[above]{$~i_j$} (14); 
        \draw[-, dashed, red] (7) to (15); 

        \node[] (17) at (3, 2) {$\xRightarrow{\text{Injection}}$};

        \node [] (18) at (4, 1) {};
		\node [] (19) at (6, 1) {};
		\node [] (20) at (7, 2) {};
		\node [] (21) at (5, 2) {};
        \node [] (22) at (6, 3) {};
        \node [] (23) at (7, 4) {};
        \node [] (24) at (5, 4) {};
        \node [] (25) at (4, 3) {};
	
		\draw[-] (18) to (19); 
        \draw[-] (19) to (20); 
        \draw[-] (19) to (22); 
        \draw[-] (18) to (25); 
        \draw[-] (18) to (21); 
        \draw[-] (20) to (23); 
        \draw[-] (20) to (21);
        \draw[-] (22) to (23); 
        \draw[-] (22) to (25); 
        \draw[-] (24) to (25); 
        \draw[-] (24) to (23); 
        \draw[-] (24) to (21); 
        \draw[->, blue, dashed] (25) to node[below]{$~p_k$} (24);

        \node[] (26) at (8, 2) {$\xRightarrow{\text{Projection}}$};

        \node [] (27) at (9, 1) {};
		\node [] (28) at (11, 1) {};
        \node [] (29) at (11, 3) {};
        \node [] (30) at (9, 3) {};

		\draw[-] (27) to (28); 
        \draw[-] (27) to (30); 
        \draw[-] (28) to (29); 
        \draw[-] (30) to (29); 

        \node[rotate=90, blue] (31) at (10, 1.5) {$\xRightarrow{\pi_\beta}$};
        \node[red] (31) at (9.3, 1.9) {$\xRightarrow{\nu_{\alpha}}$};
        \node[] (32) at (10, 2) {$({\theta_j^k})_\beta^\alpha$};
        
	\end{tikzpicture}
    \caption{Starting from a four dimensional tensor, from left to right, first apply 1-morphism injection, $i_j$ to obtain a three dimensional tensor, then apply 1-morphism projection, $p_k$ to obtain a matrix, applying 2-morphisms injection and projection, $(\pi_\beta, \nu_\alpha)$ we obtain a corresponding entry of the tensor, $({\theta_j^k})_\beta^\alpha$}
\end{figure}

To unpack Equation \ref{formula1} , consider a 2-morphism $\theta: f \Rightarrow g$ and 
\begin{equation*}
	f, g: \boxplus_{j=0}^n A_j \longrightarrow \boxplus_{k=0}^m A_k
\end{equation*}
Explicitly, $f, g$ have the following expressions:
\begin{align}
f = \bigoplus_{k=0}^m \bigoplus_{j=0}^n f_j^k(i_j \otimes p_k), && g = \bigoplus_{k=0}^m \bigoplus_{j=0}^n g_j^k(i_j \otimes p_k)
\end{align}
\begin{equation}
	\begin{tikzpicture}[scale=0.4]
		\node[] (0) at (0, 0) {$A_1 \boxplus A_2 \boxplus \dots \boxplus A_n $};
		\node[] (1) at (12, 0) {$A_1 \boxplus A_2 \boxplus \dots \boxplus A_m $};
		\node[] (2) at (6, 0.7) {};
		\node[] (3) at (6, -0.8) {};
		\draw[->, bend right=45] (0) to node[below]{$g$} (1);
		\draw[->, bend left=45] (0) to node[above]{$f$} (1);
		\draw[->, double] (2) to node[right]{~$\theta$} (3);
	\end{tikzpicture}
\end{equation}
Note that the entries of $f$ and $g$ are in the same Hom-categories as $f$ and $g$, that is,
$f_{i}^k, g_j^k \in Hom(\boxplus A_j, \boxplus A_k)$. We intentionally use $\oplus$ instead of $\sum$ to emphasis what we are summing over are 1-morphisms, so the correct sum is actually biproduct. Now if entries of $f$ and $g$ are a sum of 1-morphisms,
\begin{align}
	f_j^k = \bigoplus_{\alpha=0}^{p} h_\alpha && g_j^k = \bigoplus_{\beta=0}^{z} h_\beta 
\end{align}
The unpacked form of a 2-morphism $\theta_{j}^k : f_j^k \Rightarrow g_j^k $ is as follows:
\begin{equation}
	\theta_j^k = \sum_{\alpha=0}^{p} \sum_{\beta=0}^{z} (\theta_j^k)_\alpha^\beta(\nu_\alpha \otimes \pi_\beta)
\end{equation}
For the horizontal composition of 1-morphisms in matrix notation, we first calculate the usual multiplication of matrices, then we compose the corresponding entries of matrices. In the example above, consider a 1-morphism $r: \boxplus_{k=0}^{m} A_k \longrightarrow \boxplus_{l=0}^{w} A_l$, the composition of $r$ with $f$ is:
\begin{align*}
&f = \bigoplus_{k=0}^{m}\bigoplus_{j=0}^{n}f_j^k(i_j \otimes p_k), \hspace*{2cm} r = \bigoplus_{l=0}^{w}\bigoplus_{k'=0}^{m}r_{k'}^l(i_{k'} \otimes p_l)\\
&r\circ f = \bigoplus_{k=0}^{m} \bigoplus_{j=0}^{n} \bigoplus_{l=0}^{w} \bigoplus_{k'=0}^{m} [(f_j^k \circ r_{k'}^l)][(i_j \otimes p_k)\otimes (i_{k'}\otimes p_l)] \xRightarrow[]{p_k \otimes i_{k'} = \delta_{k, k'}}&\\
& r \circ f = \bigoplus_{j=0}^{n} \bigoplus_{l=0}^{w} \bigoplus_{k, k'=0}^{m} [(f_j^k \circ r_{k'}^l)][\delta_{kk'}(i_j \otimes p_l)] = \bigoplus_{j=0}^{n} \bigoplus_{l=0}^{w} \bigoplus_{k=0}^{m}(f_j^k \circ r_k^l)(i_j\otimes p_l)
\end{align*}
The reader should not be surprized by the tensor sign between projection and injection
$p_k \otimes i_{k^\prime}$. This convention agrees with our knowledge of tensors calculus in algebra.
The horizontal composition of 2-morphisms is similar to 1-morphisms composition.
\begin{align*}
\theta\circ \xi = \bigoplus_{j, k, m} (\theta_j^k \otimes \xi_k^m)(i_j \otimes p_m)
\end{align*}
The vertical composition is obtained first by  Hadamard product or entry-wise
matrix multiplication, and then by usual matrix multiplication between corresponding
terms. Consider 2-morphisms $f \xRightarrow{\theta} g \xRightarrow{\xi}q$, note that $f, g, q$ all have the same source
and target; 
\begin{equation*}
\centering
	\begin{tikzpicture}[scale=0.4]
		\node[] (0) at (0, 0) {$A$};
		\node[] (1) at (12, 0) {$B $};
		\draw[->, bend right=25] (0) to node[below]{$f$} (1);
		\draw[->] (0) to node[above]{$h$} (1);
        \draw[->, bend left=25] (0) to node[above]{$g$} (1);
	\end{tikzpicture}
\end{equation*}
so the same orders, i.e. order = $m\times n$. To do Hadamard product, we need to apply 1-morphism projections and injections $p_k \theta i_j = \theta_j^k$ and $p_k \xi i_j = \xi_j^k$. Then the vertical composition of $\theta$ and $\xi$ is obtained by matrix multiplication of corresponding entries. The $\sum$ sign here is the addition of 2-morphism.
\begin{align*}
	& \theta_j^k \odot \xi_j^k = \sum_{\alpha'=0}^x\sum_{\beta'=0}^z \sum_{\alpha=0}^p\sum_{\beta=0}^z  (\theta_j^k)_\alpha^\beta  (\xi_j^k)_{\beta'}^{\alpha'} [(\nu_\alpha \otimes \pi_\beta) \otimes (\nu_{\beta'} \otimes \pi_{\alpha'})] \xRightarrow{\pi_\beta \otimes \nu_{\beta'} = \delta_{\beta, \beta'}}\\
	& \theta_j^k \odot \xi_j^k = \sum_{\alpha'=0}^x \sum_{\alpha=0}^p\sum_{\beta=0}^z  (\theta_j^k)_\alpha^\beta  (\xi_j^k)_{\beta}^{\alpha'} (\nu_\alpha \otimes  \pi_{\alpha'}) 
\end{align*}
We summarize the content of this section in an example in \vect. Consider 1- and 2-morphisms in the following figure:
\begin{center}
	\begin{tikzpicture}
		\node (0) at (0, 0) {$A_1 \boxplus A_2 \boxplus A_3$};
		\node (1) at (5, 0) {$A'_1 \boxplus A'_2$};
		\node (2) at (10, 0){$A''$};
		\node[rotate=-90] (3) at (2.2, 0.5){$\Rightarrow$};
		\node (4) at (1.9, 0.5){$\theta$};
		\node[rotate=-90] (5) at (2.2, -0.5){$\Rightarrow$};
		\node (6) at (1.9, -0.5){$\eta$};
		\node[rotate=-90] (7) at (7.5, 0){$\Rightarrow$};
		\node (8) at (7.3, 0){$\xi$};  
		\draw[->, bend left=30] (0) to node[above]{$f$} (1);
		\draw[->] (0) to node[above]{$g$} (1);
		\draw[->, bend right = 30] (0) to node[below]{$l$} (1);
		\draw[->, bend left = 30] (1) to node[above]{$h$}(2);
		\draw[->, bend right=30] (1) to node[below]{$k$} (2);
	\end{tikzpicture}
\end{center}
We write 1-morphisms as matrices of orders 2 × 3 and 1 × 2. Orders of 1-morphisms
are specified by the number of objects in the target times the number of objects in
the source.
\begin{align*}
	&f=\begin{bmatrix}
		f_{11} & f_{12} & f_{13} \\
		f_{21} & f_{22} & f_{23}
	\end{bmatrix}, 
	g=\begin{bmatrix}
		g_{11} & g_{12} & g_{13} \\
		g_{21} & g_{22} & g_{23}
	\end{bmatrix}, 
	l=\begin{bmatrix}
		l_{11} & l_{12} & l_{13} \\
		l_{21} & l_{22} & l_{23}
	\end{bmatrix}, 
	&h = \begin{bmatrix}
		h_{11} & h_{12}
	\end{bmatrix}, 
	k = \begin{bmatrix}
		k_{11} & k_{12}
	\end{bmatrix} 
\end{align*}
Entries of 2-morphisms are 2-morphisms between 1-morphisms $\theta_{ij}: f_{ij} \Rightarrow g_{ij}$ written in the following. So the order of $\theta$ is the same as $f$ and $g$, but orders of entries are
determined by the number of components in entries of $f$ and $g$.
\begin{align*}
	&\theta=\begin{bmatrix}
		\theta_{11} & \theta_{12} & \theta_{13} \\
		\theta_{21} & \theta_{22} & \theta_{23}
	\end{bmatrix}, 
	\eta=\begin{bmatrix}
		\eta_{11} & \eta_{12} & \eta_{13} \\
		\eta_{21} & \eta_{22} & \eta_{23}
	\end{bmatrix},
	&\xi = \begin{bmatrix}
		\xi_{11} & \xi_{12}
	\end{bmatrix}, 
\end{align*}
Assume that some of the entries of 1-morphisms are biproducts of 1-morphisms in the same Hom-category, and other entries do not have any further components:
\begin{align*}
	& f_{11} = f_{11}^1 \oplus f_{11}^2, && g_{12} = g_{12}^1 \oplus g_{12}^2 \oplus g_{12}^3, && l_{22}=l_{22}^1 \oplus l_{22}^2, \\
	& h_{11} = h_{11}^1 \oplus h_{11}^2 && k_{11}=k_{11}^1 \oplus k_{11}^2
\end{align*}
Thus, the entries of 2-morphisms have the explicit matrix form as follows:
\begin{align*}
	& \theta_{11} = \begin{bmatrix}
		\theta_{11}^1 & \theta_{11}^2
	\end{bmatrix}, \theta_{12} = \begin{bmatrix}
		\theta_{12}^1 \\ \theta_{12}^2 \\ \theta_{12}^3
	\end{bmatrix}, \xi_{11} = \begin{bmatrix}
		\xi_{11}^1 & \xi_{11}^2 \\
		\xi_{11}^3 & \xi_{11}^4 
	\end{bmatrix}, \eta_{22} = \begin{bmatrix}
		\eta_{22}^1 \\ \eta_{22}^2
	\end{bmatrix}, 
	& \eta_{12} = \begin{bmatrix}
		\eta_{12}^1 & \eta_{12}^2 & \eta_{12}^3
	\end{bmatrix}
\end{align*}
\begin{itemize}
	\item \textit{Composition of 1-morphisms}
	\begin{align*}
		h \circ f = \begin{bmatrix}
			h_{11}f_{11} \oplus h_{12}f_{21} & h_{11}f_{12} \oplus h_{12}f_{22} & h_{11}f_{13} \oplus h_{12}f_{23}
		\end{bmatrix}
	\end{align*}
	\item \textit{Vertical composition of 2-morphisms} 
	\begin{align*}
		& \eta \odot \theta = \begin{bmatrix}
			\eta_{11} & \eta_{12} & \eta_{13} \\
			\eta_{21} & \eta_{22} & \eta_{23}
		\end{bmatrix} \odot \begin{bmatrix}
			\theta_{11} & \theta_{12} & \theta_{13} \\
			\theta_{21} & \theta_{22} & \theta_{23}
		\end{bmatrix} = \begin{bmatrix}
			\eta_{11} . \theta_{11} & \eta_{12} . \theta_{12} & \eta_{13}.\theta_{13} \\
			\eta_{21} . \theta_{21} & \eta_{22}.\theta_{22} & \eta_{23}.\theta_{23}
		\end{bmatrix} \\
		& = \begin{bmatrix}
			\begin{bmatrix}
				\eta_{11} \theta_{11}^1 & \eta_{11} \theta_{11}^2
			\end{bmatrix} & \eta_{12}^1 \theta_{12}^1 + \eta_{12}^2 \theta_{12}^2 + \eta_{12}^3 \theta_{12}^3& \eta_{13} \theta_{13}\\
			\eta_{21}\theta_{21} & \begin{bmatrix}
				\eta_{22}^1 \theta_{22} \\ \eta_{22}^2 \theta_{22}
			\end{bmatrix} & \eta_{23} \theta_{23}
		\end{bmatrix}
	\end{align*}
	\item \textit{Horizontal composition of 2-morphisms}
	\begin{align*}
		&\xi \circ \theta = \begin{bmatrix}
			\xi_{11} & \xi_{12}
		\end{bmatrix} \circ \begin{bmatrix}
			\theta_{11} & \theta_{12} & \theta_{13} \\
			\theta_{21} & \theta_{22} & \theta_{23}
		\end{bmatrix} \\
		&= \begin{bmatrix}
			\xi_{11} \otimes \theta_{11} \oplus \xi_{12} \otimes \theta_{21} &
			\xi_{11} \otimes \theta_{12} \oplus \xi_{12} \otimes \theta_{22} &
			\xi_{11} \otimes \theta_{13} \oplus \xi_{12} \otimes \theta_{23}
		\end{bmatrix} 
		&= \begin{bmatrix} \alpha
			&
			\beta
			&
			\gamma
		\end{bmatrix}
	\end{align*}
	\begin{align*}
		& \alpha = \begin{bmatrix}
			\begin{bmatrix}
				\xi_{11}^1 \begin{bmatrix}
					\theta_{11}^1 & \theta_{11}^2
				\end{bmatrix} & \xi_{11}^2 \begin{bmatrix}
					\theta_{11}^1 & \theta_{11}^2
				\end{bmatrix}\\
				\xi_{11}^3 \begin{bmatrix}
					\theta_{11}^1 & \theta_{11}^2
				\end{bmatrix} & \xi_{11}^4 \begin{bmatrix}
					\theta_{11}^1 & \theta_{11}^2
				\end{bmatrix}
			\end{bmatrix}
			\\ & \xi_{12} \otimes \theta_{21} 
		\end{bmatrix}, \\
		&\beta = \begin{bmatrix}
			\begin{bmatrix}
				\xi_{11}^1  \begin{bmatrix}
					\theta_{12}^1 \\ \theta_{12}^2 \\ \theta_{12}^3
				\end{bmatrix} & \xi_{11}^2  \begin{bmatrix}
					\theta_{12}^1 \\ \theta_{12}^2 \\ \theta_{12}^3
				\end{bmatrix}\\
				\xi_{11}^3  \begin{bmatrix}
					\theta_{12}^1 \\ \theta_{12}^2 \\ \theta_{12}^3
				\end{bmatrix} & \xi_{11}^4  \begin{bmatrix}
					\theta_{12}^1 \\ \theta_{12}^2 \\ \theta_{12}^3
				\end{bmatrix}
			\end{bmatrix} 
			\\ & \xi_{12} \otimes \theta_{22}
		\end{bmatrix} \\
		& \gamma = \begin{bmatrix}
			\begin{bmatrix}
				\xi_{11}^1  \theta_{13} & \xi_{11}^2  \theta_{13}\\
				\xi_{11}^3  \theta_{13} & \xi_{11}^4  \theta_{13}
			\end{bmatrix}\\
			&
			\\ & \xi_{12} \otimes \theta_{23}
		\end{bmatrix}
	\end{align*}
\end{itemize}
\subsection{Vectorization in 2-categories}
Analogous to the categorical case described in Section~\ref{section:vectorization1}, in this section, we intend to define a 2-categorical version of vectorization. In the 2-categorical case, we vectorize both 1-morphisms and 2-morphisms. An equation similar to Equation~\ref{eq:curry}, is defined in each Hom-category with similar properties. Let us observe the details of the vectorization procedure in both 1- and 2-morphisms with an example in the following. 

Consider two 1-morphisms; $f, f': 2 \longrightarrow 3$ with a 2-morphism between them, $\theta: f\Rightarrow f'$. Vectorization assigns two vectors, $v$ and $v'$  to 1-morphisms $f$ and $f'$ respectively. Figure~\ref{eq:currying1} also shows how this procedure assigns a vector $\alpha$ to 2-morphism $\theta$ between them. 
\begin{equation}\label{eq:currying1}
\begin{tikzpicture}
    \node(0) at (0, 2) {$3\times 2$}; 
    \node (1) at (0, 0) {$1$};
    \node (3) at (9, 0) {$3 \times 1$}; 
    \node (4) at (5, 2) {$3 \times (3 \times 2)$};
    \node (5) at (12, 2) {$2$}; 

    \draw[->, bend right=60] (1) to node[right]{$v'$}(0);
    \draw[->, bend left=60] (1) to node[left]{$v$}(0);
    \node (6) at (0, 1) {$\xRightarrow{\alpha}$}; 

    \draw[->, bend right=30] (3) to node[right]{$f$}(5);
    \draw[->, bend left=30] (3) to node[left]{$f'$}(5);
    \node[rotate=120] (7) at (10.5, 1) {$\xRightarrow{\theta}$};

    \draw[->, bend right=20] (3) to node[right]{$id_3 \otimes v'$}(4);
    \draw[->, bend left=30] (3) to node[left]{$id_3 \otimes v$}(4);
    \node[rotate=30] at (7, 1.1) {$\xRightarrow{id_3 \otimes \alpha}$};

    \draw[->] (4) to node[above]{$e_3$} (5);  
    
\end{tikzpicture}   
\end{equation}
Similar equations hold for 1-morphisms and the 2-morphism as below;
\begin{align}\label{eq:equations-currying}
    & e_2(id_2 \otimes v) = f,  e_2(id_2 \otimes v') = f',  id_{e_2} \circ(id_2 \otimes \alpha) = \theta 
\end{align}
The matrix form of the above equations for each morphism are the following, 
\begin{align*}
    & f = \begin{bmatrix}
        f_{11} & f_{12} \\ f_{21} & f_{22} \\ f_{31} & f_{32}
    \end{bmatrix}, f' = \begin{bmatrix}
         f'_{11} & f'_{12} \\ f'_{21} & f'_{22} \\ f'_{31} & f'_{32}
    \end{bmatrix}, \theta = \begin{bmatrix}
         \theta_{11} & \theta_{12} \\ \theta_{21} & \theta_{22} \\ \theta_{31} & \theta_{32}
    \end{bmatrix}
\end{align*}
where 1-morphism $e$, vector $v$ and $v'$ assigned to 1-morphisms, $f$ and $f'$ respectively, and vector $\alpha$ assigned to 2-morphism $\theta$ are written in Equation~\ref{eq:curring-matrices}. It is straightforward to examine the above equations, \ref{eq:equations-currying}, by direct substitution of matrices, Equation~\ref{eq:curring-matrices}.

\begin{equation}\label{eq:curring-matrices}
e = \begin{bmatrix}
1&0&0&0&0&0&0&1&0&0&0&0\\
0&0&1&0&0&0&0&0&0&1&0&0\\
0&0&0&0&1&0&0&0&0&0&0&1
\end{bmatrix}, v = \begin{bmatrix}
        f_{11} \\ f_{12} \\ f_{21} \\ f_{22} \\ f_{31}\\f_{32}\\0\\0\\0\\0\\0\\0\\0
    \end{bmatrix},  v' = \begin{bmatrix}
        f'_{11} \\ f'_{12} \\ f'_{21} \\ f'_{22} \\ f'_{31}\\f'_{32}\\0\\0\\0\\0\\0\\0\\0
    \end{bmatrix},\alpha = \begin{bmatrix}
        \theta_{11} \\ \theta_{12} \\ \theta_{21} \\ \theta_{22} \\ \theta_{31}\\\theta_{32}\\0\\0\\0\\0\\0\\0\\0
    \end{bmatrix}  
\end{equation}
Note that after vectorization of 2-morphisms, in general, we are faced with a column vector whose entries are matrices of various dimensions. The dimension depends on the domain and target of 2-morphism; for instance, $\theta_{11}: f_{11} \longrightarrow f_{11}'$ has a dimension $2 \times 3$ if $f_{11} = f_{11}^1 \oplus f_{11}^2 \oplus f_{11}^3$ and $f'_{11} = {f'}_{11}^1 \oplus {f'}_{11}^2$.
Hence, another internal vectorization happens for each entry of a 2-morphism. 
\[
\theta_{11} = \begin{bmatrix}
    \theta_{11}^1 & \theta_{11}^2 & \theta_{11}^3 \\
    \theta_{11}^4 & \theta_{11}^5 & \theta_{11}^6 
\end{bmatrix}
\]
Let us draw the diagram with details in the below. We stick to the dimension of source and target instead of representing 1-morphisms $f_{11}$ and $f'_{11}$ in the diagram since this representation facilitates observing the \textbf{internal currying} without losing anything specifically important. 
\begin{equation}\label{eq:currying2}
\begin{tikzcd} 
3 \times 2 &&& 3 \times (3 \times 2) \arrow[r, "e_k", Rightarrow] &2\\
1 \arrow[u, "\alpha", Rightarrow] &&& 3 \times 1 \arrow[u, "id_3 \otimes \alpha", Rightarrow] \arrow[ru, "\theta"{below}, Rightarrow] &
\end{tikzcd} 
\end{equation}
To summarise, vectorization in 2-categories follows the same principle as vectorization in categories. However, due to the existence of 2-morphisms, we encounter currying and vectorization at two levels; at the level objects, 1-morphisms and 2-morphisms summarised in Equation~\ref{eq:currying1} and at the level of 1-morphisms and 2-morphisms summarized in Equation~\ref{eq:currying2}. At level of objects, vectorization of 1-morphism and 2-morphisms follow the same principle, for each pair of vectorized 1-morphism, a vectorized 2-morphism is defined subsequently. After this level of vectorization, however, each of the entries of a 2-morphism still can have a matrix structure. Then, we vectorize 2-morphism entries with the second currying that we called \textbf{internal currying}. The internal vectorization happens in each Hom-category separately from other Hom-categories.

The recipe for vectorization in Hom-categories as presented above simply translate and represent what we had in categories. So the linear algebra algorithms using vectorization proposed by Macedo and Oliveira's carries inside our setting and Hom-categories, and all their results regarding typing of algorithms also straightforwardly translates to 1-morphisms and 2-morphisms of Hom-categories. 

\begin{rem}
In our setting, we have internal currying map for all Hom-categories; however, it is conceivable to have a 2-category that violates this. Meaning, having some Hom-categories without internal currying and vectorization. This 2-category would not be interesting from typing linear algebra perspective, but abstractly intriguing. 
\end{rem}

\section{Discussion}
In advancing the categorical framework for linear algebra, we introduced a notion of 2-biproducts within 2-categories. By proposing a limit form alongside algebraic definitions in a 2-category enriched over the 2-category of semiadditive categories, we established a definition that aligns with the 1-categorical understanding of biproducts. This structure is suitable for typing tensors with four indices as 2-morphisms and matrices as 1-morphisms. Through examples, we further illustrated the vectorization process for both 1-morphisms and 2-morphisms, underscoring the versatility of this approach.

A compelling avenue for future exploration lies in the implementation of these structures within functional programming languages such as Haskell. Furthermore, this framework holds potential for the categorical modeling of deep learning algorithms. The inherent composition and tensorial operations in this approach is similar to the architectures of modern machine learning, promising a bridge between category theory and deep learning.

The setup we considered was a (symmetric)monoidal semiadditive 2-category. For future research, we aim to investigate a more general setup of a Cartesian closed 2-category. This 2-category would have general properties and structures which we expect to change the currying we assumed for our 2-category. We, however, expect to recover our introduced correspondence from the general setup.  We leave the extension of this work to part II. 
\section{Acknowledgment}
I completed the main body of this work during my DPhil at the University of Oxford. I thank Jamie Vicary for some discussions and Carmen Constantin for proofreading the very first draft. The final draft was prepared at Imperial College London, and I thank Antoine Jack Jacquier for his support. I acknowledge support from UKRI; Grant Ref: EP/W032643/1. 

\bibliographystyle{elsarticle-num}

\bibliography{ref}

\end{document}